\documentclass{article}

\usepackage{amsmath}
\usepackage{amssymb}
\usepackage{color}
\usepackage{theorem}
\usepackage{tikz-cd}
\usepackage{cite}

\newtheorem{theorem}{Theorem}[section]
\newtheorem{proposition}[theorem]{Proposition}
\newtheorem{lemma}[theorem]{Lemma}
\newtheorem{corollary}[theorem]{Corollary}

\theoremstyle{definition}
\newtheorem{definition}[theorem]{Definition}
\theoremstyle{remark}
\newtheorem{remark}[theorem]{Remark}
\newtheorem{example}[theorem]{Example}

\newenvironment{proof}{\textbf{Proof:}}{\hfill$\square$}

\newcommand{\Stab}{\textup{Stab}} 
\title{Structural rigidity and flexibility using graphs of groups}
\author{Joannes Vermant and Klara Stokes}

\begin{document}
\maketitle
\begin{abstract}
In structural rigidity, one studies frameworks of bars and joints in Euclidean space. Such a framework is an articulated structure consisting of rigid bars, joined together at joints around which the bars may rotate. In this paper, we will describe articulated motions of realisations of hypergraphs that uses the terminology of graph of groups, and describe the motions of such a framework using group theory. Our approach allows to model a variety of situations, such as parallel redrawings, scenes, polytopes, realisations of graphs on surfaces, and even unique colourability of graphs. This approach allows a concise description of various dualities in rigidity theory. We also provide a lower bound on the dimension of the infinitesimal motions of such a framework in the special case when the underlying group is a Lie group. 
\end{abstract}

\section{Introduction}
\label{sec:intro}

An articulated motion is a motion of a structure that is assembled from several rigid parts using different kinds of joints. The research area of articulated motions has a long history and contains many important results. For example, Cauchy's theorem in geometry says that two convex polyhedra with congruent faces must be congruent.``The molecular theorem'', proved by Katoh and Tanigawa \cite{MolecularConjecture} gives a counting condition for the rigidity of panel and hinge frameworks. The dual statement in three dimensions gives a counting condition for the rigidity of  bar and revolute-joint frameworks with applications to the problem in chemistry of determining the rigidity of molecules. The Geiringer-Laman theorem \cite{PollaczekGeiringer1927,Laman1970} gives a counting condition for the rigidity of bar and pin-joint frameworks in the Euclidean plane, but the question on how to generalise it to three or more dimensions remains unsolved albeit decades of serious intents. 

 
In this article we propose to forget the space in which the articulated structure moves and instead study articulated motion using group theory only. We will model articulated structures using graphs of groups. A graph of groups is a graph with groups at the vertices and the edges, together with monomorphisms from each edge group to the groups at the end vertices of the edge. 
When two geometric objects $x$ and $y$ are incident, this phenomenon occurs naturally: the stabiliser of the incidence of $x$ and $y$ is contained as a subgroup in the stabilisers of $x$ and $y$. 

We suggest the combinatorial study of motions of hypergraphs realised as a graph of groups. We now give a brief motivation of the approach, to make the ideas more easily digestible. To model motions, we will make use of the set of incidences of a graph $\Gamma=(V,E)$, this is the set
\begin{equation*}
   I =  \{\{v,e\} \subseteq V\cup E ~ \vert ~ v\in e\},
\end{equation*}
and we will use the notation $v*e$ instead of $\{v,e\}$ for incidences.
Suppose now that we have two equivalent realisations of a graph $\Gamma=(V,E)$ in the Euclidean plane, $p_1, p_2: V\rightarrow \mathbb{R}^{2}$, where equivalent means that the Euclidean distance between any two adjacent vertices are the same in both realisations. Then, for all $v\in V$, $e\in E$ and $i=v*e\in I$, there are then group elements $g_{x}\in SE(2)$ such that 
\begin{equation*}  
\begin{array}{ccc}
    g_v p_1(v)= p_2(v),& & g_e p_1(e) = p_2(e)\\
   g_i p_1(v)= p_2(v),& & g_i p_1(e) = p_2(e),
\end{array} 
\end{equation*}
where $p_n(v)$ is the point assigned to the vertex $v$ and $p_n(e)$ is the oriented line segment assigned to the edge $e$ for $n\in \{1,2\}$. See Figure \ref{fig:motion} for an illustration.

\begin{figure}
\begin{center}
    \begin{tikzpicture}
    \node [] (P) at (-0.5,3) {{$v$}};
    \node [] (L) at (1,2.75) {{$e$}};
    \node [] (L') at (0.2,3.5) {{$i=v*e$}};
    \node [circle,fill,inner sep=2pt] (A) at (0,0) {};
    \node [circle,fill,inner sep=2pt] (B) at (0,3) {};
    \node [circle,fill,inner sep=2pt] (C) at (2,3) {};
    \node [circle,fill,inner sep=2pt] (D) at (2,0) {};
    \node[ draw=none] (mid1) at (1,2.95) {};
    \node[ draw=none] (I) at (0.5,3) {};
    \node[ draw=none] (begin1) at (0.2,2.95) {};
    \draw [line width=1pt] (A) -- (B) -- (C) -- (D) --(A);
    \node [circle,fill,inner sep=2pt] [right] (D) at (4,0) {};
    \node [circle,fill,inner sep=2pt] (E) at (5.4383,2.6327) {};
    \node [circle,fill,inner sep=2pt] (F) at (7.4383,2.6327) {};
    \node [circle,fill,inner sep=2pt] (G) at (6,0) {};
    \node[ draw=none] (mid2) at (6.4383,2.58) {};
    \node[ draw=none] (begin2) at (5.8383,2.58) {};
    \draw [line width=1pt] (D) -- (E) -- (F) -- (G) --(D);
    \path[]
        (B) edge[draw=black, ->, bend right] node [right] [below,midway] {\textcolor{black}{$g_v$}} (E)
        (mid1) edge[draw=black, ->, bend left] node [left] [above,midway]{\textcolor{black}{$g_e$}} (mid2)
        (begin1) edge[draw=black, ->, bend left] node [left] [below, midway]{\textcolor{black}{$g_i$}} (begin2);
    \end{tikzpicture}
    \caption{Part of a motion between two realisations}\label{fig:motion}
        \end{center}

    \end{figure}
The equations above imply that $g_i^{-1} g_v$ stabilises $p_1(v)$ and $g_{i}^{-1}g_e$ stabilises $p_1(e)$. The group elements $g_x$ exist precisely because the two realisations $p_1$ and $p_2$ are equivalent, and conversely group elements satisfying these conditions yield equivalent frameworks, as we show in Section \ref{sec: bar-joints}. Thus we may forget the geometric realisations $p_1$ and $p_2$; remembering only the stabilisers we can still describe the equivalent frameworks to a bar-joint framework, using only the graph and these groups. 

We model the rigidity problem of a realisation of a hypergraph $\Gamma=(V,E)$ in a group $G$ by assigning subgroups of $G$ to the vertices and edges of the incidence graph of $\Gamma$, thereby defining the realisation of a hypergraph in terms of a graph of groups. We define a motion of a graph of groups, and say what it means for such a realisation to be rigid. 

We believe that our new approach to structural rigidity will serve as motivation for combining rigidity theory with geometric group theory. The article is written without assuming previous knowledge about geometric group theory.
Our work is motivated by the work of many other authors who studied the rigidity of articulated structures in different spaces.                       

Gortler, Gotsman, Liu and Thurston developed a theory for structural rigidity in affine space \cite{AffineRigidity} in terms of a monoid $M$ acting on Euclidean space $\mathbb{R}^{n}$. 
They defined a framework of a hypergraph $\Gamma= (V,E)$ to be a function $p: V\rightarrow \mathbb{R}^{n}$. 
Two  frameworks $p,q: V\rightarrow \mathbb{R}^{n}$ are said to be pre-equivalent if for every hyperedge $e$ there exists a $g_e\in M$ such that  
\begin{equation}\label{affine-constraint.}
    g_e p(v)=q(v) \text{ for every } v\in e.
\end{equation}
Two frameworks $p,q:V\rightarrow \mathbb{R}^{n}$ are said to be congruent if there exists a $g\in M$ such that 
\begin{equation}\label{affine-constraint2}
    g p(v)=q(v) \text{ for every } v\in V.
\end{equation}
A framework is then called rigid if there exists a neighbourhood of $p$ such that every framework which is pre-equivalent to $p$ is congruent to $p$. When $M$ is a group, the notion of pre-equivalence  becomes an equivalence relation. 

Scherck and Mathis \cite{Decomposition} have considered various symmetry groups inside of the Euclidean group associated to different CAD constraints, which they used to develop a decomposition algorithm. In their model, if the angle between two lines is constrained to be conserved, then this pair of lines would get the group of similarities associated to it. If instead the distance between two points were to be constrained, then one would consider the Euclidean group. Based on the type of constraint (angle conservation between lines, distance constraint between points, incidence conservation, etc.) the group is varied. The authors developed an algorithm to decompose the system into various subsystems, for which it is possible to check unique realisability (up to a group action).

Stacey, Mahony, and Trumpf \cite{SymmetryinRigidity}, give general definitions for rigidity using topological groups. Their setup is as follows. The set of frameworks is a Hausdorff topological space $\mathcal{M}$, with a topological group $G$ acting continuously on $\mathcal{M}$. We are given an invariant function
\begin{equation*}
    h:\mathcal{M}\rightarrow \mathcal{Y},
\end{equation*}
where $\mathcal{Y}$ is some topological space, whose elements are thought of as sensor measurements. Then, a configuration is defined as a fixed state $x\in \mathcal{M}$ and a formation is defined as a fiber $\mathcal{F}(y)=h^{-1}(y)$. Then, if $\mathcal{F}(y)$ is $G$-invariant, $\mathcal{F}(y)$ is said to be locally rigid if for every $x\in \mathcal{F}(y)$ there exists an open set $U_x\subset \mathcal{M}$ such that any $y\in U_x \cap \mathcal{F}(y)$  is congruent to $x$. If $U$ can be taken to be $\mathcal{M}$, $\mathcal{F}(y)$ is called globally rigid. A notion of path rigidity was introduced as well, which informally says that there is a path in the group moving one framework in a formation to any other formation.

The first two authors later extended these notions to Lie groups in \cite{stacey2016generalised}, with an extension of infinitesimal rigidity as well. In this article, the space $\mathcal{Y}$ is assumed to be a manifold, and the measurement map $h$ is assumed to be smooth. Infinitesimal motions are given by elements in $\ker(dh)$, where $dh$ is the differential of the map, and the framework is infinitesimally rigid if $\dim(\ker(dh))$ has the same dimension as the group.

In his thesis, and later together with several authors, Dewar studied the rigidity of frameworks in non-Euclidean normed spaces \cite{dewarthesis}.  For example, in the paper \cite{NormedSpaces}, Lie groups are  used to prove an equivalence between infinitesimal rigidity and local rigidity. Furthermore, the bounds on the dimension of the trivial infinitesimal motions of frameworks in normed spaces comes from the Lie group.

A combinatorial characterisation of rigidity of generic bar and joint frameworks realised on surfaces with non-trivial space of infinitesimal motions was given by Nixon, Owen and Power \cite{NiOwPo2012,NiOwPo2014}.
See also work by 
Cruickshank, Guler, Jackson, and Nixon, where they also characterise rigidity of bar-joint frameworks under the assumption that certain vertices are restricted to lie in an affine subspace \cite{Linearconstraints}.

The article is structured as follows. In Section \ref{sec:bas_def}, we establish notation and define the concepts that are needed later in the paper. Section \ref{sec:Defs_hypergraphs} contains the main ideas of the paper. In Section \ref{sec:examples} we show the usefulness of our approach by applying it to various examples. 

\section{Preliminaries}
\label{sec:bas_def}

\subsection{Group theory and graphs of groups}\label{sec:bas_def_groups}
Let $G$ be a group acting on a set $X$. The stabiliser of an element $x\in X$ is the subgroup $\textup{Stab}(x)= \{g\in G~\vert~ g \cdot x = x\}$. 
Given a subgroup $H\leq G$ we denote the action of conjugation by $H^{\sigma}= \sigma H \sigma^{-1}$, and the normaliser of $H$ is the subgroup $$N_G(H):= \{g\in G~\vert ~ H^g = H \}\leq G.$$ 
A subgroup is said to be self-normalizing if $N_G(H)=H$.
If a group $G$ acts on a space $X$, then for a subgroup $H\leq G$, the fix of $H$ is defined as
\begin{equation*}
    \textup{Fix}(H):= \{x\in X~\vert~ g\cdot x=x \text{ for all } g\in H\}.
\end{equation*}

A graph of groups is a graph $(V,E)$ with a group $H_v$ associated to every vertex $v\in V$ and to every edge $e=(u,v)\in E$, together with two monomorphisms from every edge group $H_{(u,v)}$ to the two groups $H_u$ and $H_v$ at its end-points.
The fundamental group $\pi_1$ of a graph of groups can be computed using amalgams and HNN-extensions from the vertex groups over the edge groups.
The structure theorem says that $\pi_1$ acts on a universal covering tree $X$ of the graph $(V,E)$ and that the quotient $X/\pi_1$ is isomorphic to $(V,E)$. When the vertex groups are trivial, then $X$ is simply the universal cover of the graph and $\pi_1$ is the first homotopy group of the graph.

In the spirit of Tits, who started the tradition of studying groups by studying their action on incidence geometries, Bass and Serre introduced graphs of groups in the 1970's, inspired by Ihara's proof that every torsion-free discrete subgroup of $\mathbb{SL}_2(\mathbb{Q}_p)$ is a free group \cite{Serre1980}. Their original motivation was to study algebraic groups whose Bruhat-Tits buildings are trees. Today Bass-Serre theory is fundamental in geometric group theory. Important information about the group can be obtained by letting it act on trees, and more generally, on buildings and diagram geometries.

\subsection{Lie groups} \label{sec:bas_def_lie}
For a general reference on Lie groups, see \cite{Kirillov} and references therein. 

If $X$ is a manifold and $p\in X$, then $T_pX$ denotes the tangent space at $p$. Given a smooth map $\phi: X\rightarrow Y$, we denote the pushforward (or differential) by $d\phi_p : T_pX\rightarrow T_{\phi(p)}Y$. Whenever a map has a subscript, say for example $f_1:X\rightarrow Y$, we will write $d(f_1)_{p}$ to avoid confusion between the index $1$, and the point $p$ whose tangent space is the domain of the differential.  We recall that a map is an immersion if $d\phi_p$ is injective for all $p\in X$. A map is called an embedding if it is an immersion and a homeomorphism onto its image, when considering the image as a subspace of the codomain. 
We recall that a Lie group is a group $G$ which is also a smooth manifold, such that the multiplication and inverse maps are smooth functions. 

\begin{definition}
A subgroup $H$ of a Lie group $G$ is called a Lie subgroup if $H$ is an immersed submanifold of $G$. A subgroup $H$ of a Lie group $G$ is called a closed Lie subgroup if $H$ is an embedded submanifold of $G$. 
\end{definition}

\begin{theorem} {\cite[Theorem 2.9]{Kirillov} }
A subgroup of a Lie group $G$ which is closed (topologically) is a closed Lie subgroup.
\end{theorem}

To every Lie group $G$ there is an associated Lie algebra, which we denote by $\mathfrak{g}$. It can be identified with the tangent space at the identity, or with the space of left invariant vector fields on $G$. If $H \leq G$ is a Lie subgroup, then its associated Lie algebra $\mathfrak{h}$ is a subalgebra of $\mathfrak{g}$. In general we denote the Lie algebra of a Lie group by the corresponding fraktur font letter. 

\begin{theorem}\label{OrbitsKir}{\cite[Theorem 3.29]{Kirillov}}
Let $G$ be a Lie group, acting on a manifold $M$, through the function;
\begin{equation*}
    \rho: G \times M \rightarrow M
\end{equation*}
and let $\mathfrak{g}$ be the Lie algebra of $G$. Let $p\in M$. Then 
\begin{itemize}
    \item the stabiliser $G_p$ of $p$ is a closed Lie subgroup with Lie algebra $\mathfrak{g}_p=\{x\in \mathfrak{g}~\vert~ d\rho(x)p=0\}$, where $d\rho(x)$ is the vector field on $M$ corresponding to $x$. In other words $d\rho(x)$  is the pushforward of the left invariant vector field associated with $x$, 
    \item{the map $G/G_p \rightarrow M$ is an immersion. The image is the orbit of $p$ under the action. Thus the orbit $O(p)$ is an immersed submanifold of $M$, with tangent space $T_p(O(p))=\mathfrak{g}/\mathfrak{g}_p$}.
\end{itemize}
\end{theorem}

\begin{theorem}{\cite[Theorem 2.11]{Kirillov} }\label{fiber_bundle}
Let $G$ be a $n$-dimensional Lie group, and let $H\leq G$ be a closed Lie subgroup of dimension $k$. Then $G/H$ has a natural structure of a smooth manifold of dimension $n-k$ such that
\begin{equation*}
    \pi: G \rightarrow G/H 
\end{equation*}
is a fibre bundle, with fibre diffeomorphic to $H$. 
\end{theorem}

\subsection{Rigidity theory}\label{sec:bas_def_rig}
For a reference on local rigidity, infinitesimal rigidity, and how they relate, we refer the reader to Asimow and Roth's paper \cite{asimow1978rigidity}. For a reference on global rigidity, we refer to the survey by Whiteley and Jord\'an \cite{jordan2017global}. A thorough overview over how projective geometry can be used in rigidity theory is given in the recent paper by Nixon, Schulze and Whiteley \cite{RigProjlens}.

\begin{definition}
    A $d$-dimensional bar-joint framework $(\Gamma,p)$ consists of a graph $\Gamma=(V,E)$ together with a function $p:V\rightarrow \mathbb{R}^d$. Two bar-joint frameworks $(\Gamma,p)$ and $(\Gamma,q)$ are said to be equivalent if for every edge $vw$, one has
    \begin{equation*}
        \Vert p(v) - p(w)\Vert =\Vert q(v) - q(w)\Vert. 
    \end{equation*}
Two bar-joint frameworks $(\Gamma,p)$ and $(\Gamma,q)$ are said to be congruent if there exists a Euclidean transformation $\phi$ such that
    \begin{equation*}
        \phi( p(v))= q(v)
    \end{equation*}
The set of all bar-joint frameworks of a given graph $\Gamma=(V,E)$ can be identified with $\mathbb{R}^{d\vert V\vert}$, by considering $(x_1, \dots, x_{\vert V\vert})$, with each $x_i\in \mathbb{R}^{d}$ to be the bar-joint framework given by $p(v_i) := x_i$.
A flex of a bar-joint framework $p$ is a continuous curve $$\gamma: [0,1] \rightarrow \mathbb{R}^{d\vert V\vert}: t\mapsto (x_{1}(t), \dots x_{\vert V\vert}(t)),$$
such that $\gamma(0)=p$, and for each $t$, $\gamma(t)$ defines a bar-joint framework which is equivalent to $p$. 
\end{definition}

\begin{definition}
    A $d$-dimensional bar-joint framework $(\Gamma,p)$ is 
    \begin{enumerate}
        \item locally rigid if for any flex $\gamma$ of $p$, the frameworks $p$ and $\gamma(t)$ are congruent for all $t\in [0,1]$.

 \item globally rigid if any bar-joint framework $(\Gamma,q)$ that is equivalent to $(\Gamma,p)$, also is congruent to $(\Gamma,p)$, and 
 \item locally/globally flexible if it is not locally/globally rigid.
 \end{enumerate}
\end{definition}

In the literature, there are several other equivalent definitions of local rigidity. Local rigidity is a generic property, in the sense that either almost all frameworks of $\Gamma$ are rigid, or almost all frameworks of $\Gamma$ are flexible. See for instance \cite{asimow1978rigidity}. 
Local rigidity is often replaced by the more tractable notion of infinitesimal rigidity; a linearisation of local rigidity.  

\begin{definition}
Let $(\Gamma, p)$ be a $d$-dimensional bar-joint framework. Label the vertices $V=\{v_1,\dots,v_{\vert V\vert}\}$. Let 
\begin{equation*}
    \begin{array}{rccc}
    f_\Gamma: &\mathbb{R}^{d\vert V\vert } &\rightarrow & \mathbb{R}^{\vert E \vert}\\\\
    &(x_{(v_1,1)}, x_{(v_1,2)}, \dots, x_{(v_{\vert V\vert},d)}) &\mapsto & \left(\sum_{i=1}^{d}(x_{(v,i)}- x_{(w,i)})^{2}\right)_{(v,w)\in E}.
    \end{array}
\end{equation*}
Elements of $\ker(d(f_\Gamma)_p)$ are called infinitesimal motions of $(\Gamma,p)$. Let $k=\dim(\textup{Span}(p(v)_{v\in V}))$. If 
\begin{equation*}
    \dim(\ker(d(f_\Gamma)_p)= \binom{d+1}{2} - \binom{d-k}{2},
\end{equation*}
then we say that $(\Gamma,p)$ is infinitesimally rigid.
\end{definition}

Already in 1864, Maxwell gave a necessary condition for rigidity of a bar-joint framework in two or three dimensions.
In three dimensions Maxwell's rule says that
$m-s= 3j-b-6$
where $b$ is the number of bars, $j$ is the number of joints, $m$ is the number of infinitesimal internal mechanisms and $s$ is the dimension of self-stresses. 
If  $m=s=0$ then the framework is kinematically and statically determinate or isostatic; minimally infinitesimally rigid and maximally stress-free.
Such frameworks have no motion, they are (minimally) rigid. 

A consequence of Maxwell's rule is the following lower bound on the dimension of the infinitesimal motions of the bar-joint framework. 
\begin{equation}
    \dim(\ker(d(f_\Gamma)_p)) \geq d \vert V \vert - \vert E\vert
\end{equation}

From Maxwell's rule, a necessary condition for generic rigidity follows. Namely, if $\Gamma$ is generically minimally (locally) rigid in $d$-dimensional Euclidean space, then
\begin{equation*}
\begin{array}{lcl}
\vert E\vert &=&  d\vert V\vert- \binom{d+1}{2}\medskip\\
\vert F\vert &\leq&  d\vert V(F)\vert- \binom{d+1}{2} \textup{ for every subset of edges $F\subseteq E$}.
\end{array}
\end{equation*}

In one and two dimensions, Maxwell's count  is also sufficient: a two-dimensional framework is minimally rigid in generic position if and only if $|F|\leq 2|V(F)|-3$  for all subsets $F\subseteq E$ of the graph $\Gamma=(V,E)$ and additionally $|E|=2|V|-3$. This was shown in  1927 by Geiringer and in 1970 by Laman \cite{PollaczekGeiringer1927, Laman1970}. 

\section{Realising hypergraphs as graphs-of-groups}\label{sec:Defs_hypergraphs}

\subsection{Graph-of-groups realisations and their motions}
    For any group $G$, let $S(G)$ be the lattice of subgroups of $G$.
    Given a hypergraph $\Gamma = (V,E)$, a realisation of $\Gamma$ as a graph of groups in a group $G$ (or a graph-of-groups realisation, or simply realisation) is a function
    \begin{equation*}
        \begin{array}{rccl}
        \rho:& V\cup E &\rightarrow &S(G)\\ 
        &x&\mapsto &\rho(x).
        \end{array}
    \end{equation*}

    This yields a graph of groups on the incidence graph of $\Gamma$.
    \begin{definition}
    Let $\Gamma=(V,E)$ be a hypergraph. The incidence graph, denoted by $I(\Gamma)$ is the graph with vertex set $V(I(\Gamma))= V\cup E$ and edge set given by the set of incidences 
    \begin{equation*}
       E(I(\Gamma))= \{\{v,e\}\in V\cup E~\vert ~v\in e\}.
    \end{equation*}
    We will denote an incidence $\{v,e\}$ by $v*e$, and the set of incidences by $I$.
    \end{definition}

    By taking the incidence graph of $\Gamma$ and putting the group $\rho(v)$ on the vertex representing the vertex $v\in V$, $\rho(e)$ on the vertex representing the edge $e\in E$, and $\rho(x*y)=\rho(x)\cap \rho(y)$ at the edge representing the incidence $x*y$, we indeed obtain a graph of groups. The monomorphisms 
    \begin{align*}
        i_x:&\rho(x*y) \rightarrow \rho(x)\\
        i_y:&\rho(x*y) \rightarrow \rho(y),
    \end{align*}
    are defined in a natural way by inclusion. We will use the notation $\rho(\Gamma)$ to denote a hypergraph together with a realisation $\rho$ as a graph of groups.

\begin{definition}
    \label{def:motion}    
    A  motion of a realisation $\rho(\Gamma)$ in a group $G$ is an indexed set of group elements $(\sigma_x)_{x\in V\cup E \cup I} \in G^{V\cup E \cup I}$, such that for every $x\in V\cup E$, and $i\in I$ with $i=x*y$ or $i=y*x$ one has that 
    \begin{equation}
    \label{eq:incidencepreserving}
        \sigma_i^{-1}\sigma_x\in \rho(x).
    \end{equation}
    \end{definition}

    A motion $M=(\sigma_x)_{x\in V\cup E \cup I}$ of a realisation $\rho(\Gamma)$ defines a new realisation $\rho^{M}(\Gamma)$, obtained by conjugating the groups $\rho(v)$ and $\rho(e)$ with the corresponding group elements in the motion. 
    In other words, if $\rho$ is a realisation, then  $\rho^{M}$ is defined as
    $$\begin{array}{rccl}
        \rho^{M}:& \Gamma &\rightarrow &S(G)\\
        &v&\mapsto &\rho(v)^{\sigma_v} \text{ for vertices}, \\ 
        &e&\mapsto &\rho(e)^{\sigma_e} \text{ for edges}. \\ 
    \end{array}$$

    Note that for any realisation $\rho(\Gamma)$, and any $g\in G$, the tuple $(g , \dots , g)$ always defines a motion of $\rho(\Gamma)$, since $g_i^{-1} g_x= g^{-1} g=\textup{id}$, which is contained in $\rho(x)$ no matter the choice of $\rho(x)$. We call such motions {\em trivial motions}. 

    \begin{definition}
     Two motions $M_1, M_2$ are equivalent, denoted by $M_1 \sim M_2$, if there exists a motion $M=(\sigma_{v_1}, \dots, \sigma_{i_m})$ with $\sigma(x)\in \rho(x)$ for all $x\in V\cup E \cup I$ such that $M_1 = M_2 \circ M$.
    \end{definition}
    If $M_1$ and $M_2$ are equivalent, then $\rho^{M_1} =\rho^{M_2}$. The converse is not always true and in Proposition \ref{normalizer-lemmaMotions} we give conditions for this to be true. 
    
    \begin{definition}\label{congruent}
    Let $\Gamma=(V,E)$ be a hypergraph. Define two realisations $\rho_1(\Gamma), \rho_2(\Gamma)$ to be congruent if there exists some $\sigma\in G$ such that for every $x\in V\cup E$ one has
    \begin{equation}
    \label{eqcon}
        \rho_1(x)^{\sigma}=\rho_2(x).
    \end{equation}
    In other words, $\rho_1$ and $\rho_2$ are congruent if there exists a trivial motion $M$ with $\rho_1^{M} = \rho_2$
    \end{definition}

    \begin{definition}[Global rigidity]\label{globallyrigidhypergraph}
    A realisation of a hypergraph $\rho(\Gamma)$ is globally rigid if for every motion $M$, $\rho^{M}$ is congruent to $\rho$.
    \end{definition}

\begin{proposition}
    The set of all graph-of-group realisations in $G$, together with motions of $\rho(\Gamma)$ defines a groupoid $\mathcal{M}$. More precisely, $\mathcal{M}$ has object set 
    $$ob(\mathcal{M})= \{\rho:  \Gamma \rightarrow S(G)\}$$
    and sets of morphisms
    $$\mathcal{M}(\rho_1, \rho_2) := \left\{ M\in G^{V\cup E \cup I}~\vert~ M \mbox{ is a motion of }\rho_1(\Gamma)\mbox{ with } \rho_1^{M} = \rho_2\right\}.$$
    The composition of two morphisms (i.e. motions) $M_1=(\sigma_x)_{x\in V\cup E \cup I}:\rho_1\rightarrow \rho_2$ and $M_2=(\tau_x)_{x\in V\cup E \cup I}: \rho_2 \rightarrow \rho_3$ is given by  $$M_2\circ M_1= (\tau_x\sigma_x)_{x\in V\cup E \cup I}: \rho_1 \rightarrow \rho_3.$$ 
\end{proposition}
\begin{proof}
It is clear that for every object there is an identity element, namely $(\textup{id}, \dots, \textup{id})$. Every motion $M:\rho_1 \rightarrow \rho_2$ given by $(\sigma_{x})_{x\in V\cup E \cup I}$ has an inverse $M^{-1}: \rho_2 \rightarrow \rho_1$ given by $(\sigma^{-1}_{x})_{x\in V\cup E \cup I}$, and composition is clearly associative. It is easy to check that the inverse is a motion. 
We will now verify that for any two motions $(\sigma_x)_{x\in V\cup E \cup I} = M_1:\rho_1 \rightarrow \rho_2$, and $ (\tau_{x})_{x\in V\cup E \cup I} = M_2: \rho_{2} \rightarrow \rho_3$, the composition $M_2\circ M_1$
is a motion with $\rho_1^{M_2\circ M_1} =\rho_3$. Indeed, for every incidence $i= x*e$ one has
\begin{align*}
(\tau_i \sigma_i)^{-1} \tau_x \sigma_x &= \sigma_i^{-1}\tau_i^{-1} \tau_x \sigma_x\\
&=\sigma_i^{-1} g' \sigma_x\mbox{ for some } g'\in \rho_2(x)\\
&=\sigma_i^{-1} \sigma_x g_0\sigma_x^{-1} \sigma_x\mbox{   for some } g_0\in \rho_1(x)\\
&= \sigma_i^{-1} \sigma_x g_0\mbox{   for some } g_0\in \rho_1(x),
\end{align*}
and since $\sigma_i^{-1}\sigma_x\in \rho_1(x)$ we get $(\tau_i \sigma_i)^{-1} \tau_x \sigma_x \in \rho_1(x)$, so the composition is a well defined motion of $\rho_1$. We have $\rho_1^{M_2\circ M_1}=\rho_3$, since 
$$    \rho_1(x)^{\tau_x\sigma_x}=(\rho_1(x)^{\sigma_x})^{\tau_x}= \rho_2(x)^{\tau_x} = \rho_3(x).
$$
\end{proof}\\

From any group homomorphism one obtains a new graph of groups. In particular, when one has a group isomorphism this yields an isomorphism between the sets of motions of these two structures.  In Section \ref{sec:examples} we will apply this to show that the liftings of scenes are dual to parallel redrawings.

\begin{proposition}\label{GroupHom}
    Let $f: G\rightarrow H$ be a group homomorphism and let $\rho_G$ be a realisation of a hypergraph $\Gamma$ as a graph of groups in $G$. 
    Then $\rho_G$ induces a realisation $\rho_{H}$ of $\Gamma$ as a graph of groups in $H$, and this gives rise to a morphism of groupoids
    \begin{equation*} 
    f_*: \mathcal{M}_G \rightarrow \mathcal{M}_H\end{equation*}
where $\mathcal{M}_G$ is the groupoid of motions of $\Gamma$ in $G$, and $\mathcal{M}_H$ is the groupoid of motions of $\Gamma$ in $H$. On objects it is defined by $$f_*(\rho)(x)= f(\rho(x))$$ and for any morphism $s=\{\sigma_x\}_{x\in V\cup E \cup I}$ from $\rho_1$ to $\rho_2$, we define  $$f_*(s) :=\{f(\sigma_x)\}_{x\in V\cup E \cup I}.$$  
Moreover this is functorial in the sense that $f_* \circ g_* =(f\circ g)_*$ and $\textup{id}_* = \textup{id}$. This implies that if $f$ is an isomorphism, then so is $f_{*}$.
\end{proposition}
\begin{proof}
It is clear that the morphism $f_*: \mathcal{M}\rightarrow \mathcal{M}'$ is well defined on objects. Let us check that given a motion $\{\sigma_x\}_{x\in V\cup E\cup I}$ from $\rho_1$ to $\rho_2$,  $\{f(\sigma_x)\}_{x\in V\cup E\cup I}$ defines a motion, since
\begin{equation*}
    f(\sigma_i)^{-1}f(\sigma_x)=f(\sigma_i^{-1}\sigma_x)\in f(\rho(x)).
\end{equation*}
This is a motion from $f_*(\rho_1)$  to $f_*(\rho_2)$, since
        \begin{equation*}
    f_*(\rho_1)(x)^{f(\sigma_x)} =  f(\rho_1(x))^{f(\sigma_x)}=f(\rho_1(x)^{\sigma_x})=f_*(\rho_2(x)).\end{equation*}
Finally, the functorality follows almost immediately from the definitions.
\end{proof}\\

In the examples in Section \ref{sec:examples}, we will see that the role of the vertices and the edges may be swapped under the isomorphism. In our examples, this happens because the group isomorphism comes from a duality of projective space. Since the definitions of a motion and of a graph of groups realisation are symmetric in the roles of vertices and edges, the vertices can be relabelled as edges and vice versa. To make this formal, one needs the concept of a dual hypergraph. 

Let $\Gamma=(V,E)$ be a hypergraph. For any $v\in V$, let $E_{v}:=\{ e\in E~ \vert ~ v\in e \}$. We define the dual hypergraph $\Gamma^*=(V^{*},E^*)$, where 
\begin{align*}
    V^{*}(\Gamma) &= E,\\
    E^{*}(\Gamma) &= \{ E_{v} \subset V^{*}~\vert ~ v\in V\}.
\end{align*}
Here $E^{*}(\Gamma)$ is considered to be a multi-set, where we have distinct copies of $E_{v}, E_{w}$ if $v\neq w$ even if $E_v=E_w$. It is easy to prove the following lemma.
\begin{lemma}\label{Dualhypergraph}
Let $\Gamma=(V,E)$ be a hypergraph and let $G$ be a group.  The groupoid $\mathcal{M}_\Gamma$ with objects 
\begin{align*}
    ob(\mathcal{M}_\Gamma)=\{\rho: \Gamma \rightarrow S(G)  \}
\end{align*}
and morphisms
\begin{align*}
    \mathcal{M}_{\Gamma}(\rho_1, \rho_2)= \left\{ M=(\sigma_x)_{x\in V\cup E \cup I} ~\vert~ M \mbox{ is a motion of }\rho_1(\Gamma)\mbox{ with } \rho_1^{M} = \rho_2\right\}.
\end{align*}
is isomorphic to the groupoid  $\mathcal{M}_{\Gamma^{*}}$, with objects
\begin{align*}
    ob(\mathcal{M}_{\Gamma^{*}})=\{\rho: \Gamma^{*} \rightarrow S(G)  \}
\end{align*}
and morphisms
\begin{align*}
    \mathcal{M}_{\Gamma^*}(\rho_1, \rho_2)= \left\{ M:=(\sigma_x)_{x\in V^{*}\cup E^{*} \cup I} ~\vert~ M \mbox{ is a motion of }\rho_1(\Gamma^{*})\mbox{ with } \rho_1^{M} = \rho_2\right\}.
\end{align*}
\end{lemma}

In most of the examples in Section \ref{sec:examples}, the subgroups $\rho(x)$ will be taken to be the stabiliser of a geometric object under a transitive group action. When the subgroup $\rho(x)$ for every $x\in V\cup E$ is self-normalising, one can recover this geometric object associated to $x$, since in this case $\textup{Fix}(\rho(x))=x$. This is made precise by the following lemma.

\begin{lemma}\label{normalizer-lemma}
Suppose we are given a group $G$ which acts transitively on a set $X$. The stabilisers are self-normalising (i.e. $N_G(\Stab(y))= \Stab(y)$), if and only if the stabilisers correspond uniquely to elements of $X$, in the sense that
\begin{equation*}
 \textup{Fix}(\Stab(y)) = \{x\in X~\vert~ h x =x \text{ for all } h\in \Stab(y)\}= \{y\}.
\end{equation*}
\end{lemma}
\begin{proof}
We will prove this by showing $\textup{Fix}(\Stab(y)) = N_G(\Stab(y))\cdot y$.

First, let $x\in \textup{Fix}(\Stab(y))$. By transitivity, there exists a $g\in G$ such that $ g \cdot x= y$. Since $x\in \textup{Fix}(\Stab(y))$, for any $h\in \Stab(y)$, we have $h \cdot x = x$, we thus see that $gh g^{-1} \cdot y= y$, thus $ghg^{-1}\in \Stab(y)$, and thus  $g\in N_G(\Stab(y))$, which implies that  $x\in N_G(\Stab(y)) \cdot y$. Conversely, let $x\in N_G(\Stab(y)) \cdot y$, then $x=g \cdot y$ for some $g$, with $ghg^{-1}\in \Stab(y)$. Take $h\in \Stab(y)$. Then $h \cdot x = h g\cdot y= g h'\cdot y $ for some $h' \in \Stab(y)$, so $h \cdot x = g \cdot y = x$, and hence $x\in \textup{Fix}(\Stab(y))$.

Let us now prove that if $\textup{Fix}(\Stab(y))=\{y\}$, then $\Stab(y)$ is self-normalising. We note that $N_G(\Stab(y)) \cdot \textup{Fix}(\Stab(y)) = \textup{Fix}(\Stab(y))$, by the above. Thus if 
\begin{equation*}
   \textup{Fix}(\Stab(y))= \{y\},
\end{equation*}
We see that $N_G(\Stab(y))\subseteq \Stab(y)$, and since the other inclusion is always true, we see that these sets are equal. 

We now show the converse. Suppose $N_G(\Stab(y))=\Stab(y)$. Pick any $z\in \textup{Fix}(\Stab(y))$, then $z\in  N_G(\Stab(y)) \cdot y =\{y\}$, and thus $z=y$.
\end{proof}\\

We now give a condition for when $\rho^{M_1} = \rho^{M_2}$ implies that $M_1\sim M_2$. If $\rho^{M_1} = \rho^{M_2}$, then $\rho^{M_1^{-1}\circ M_2} = \rho$, and thus it suffices to compute the isotropy subgroup 
\begin{equation*}
    G_{\rho}=\{M: \rho \rightarrow \rho~ \vert ~ M \text{ is a motion }\}.
\end{equation*}
Indeed, if $G_\rho = \prod_{x\in V\cup E \cup I} \rho(x)$, then $M_1^{-1}\circ M_2\in \prod_{x\in V\cup E \cup I} \rho(x)$, and hence $M_1\sim M_2$.

\begin{proposition}\label{normalizer-lemmaMotions}
    Let $\rho$ be a graph-of-groups realisation in a group $G$, and let $G_{\rho}$ be the isotropy subgroup of $\rho$. Then $G_{\rho}= \prod_{x\in V\cup E \cup I} \rho(x)$ if one of the following conditions holds.
    \begin{enumerate}
        \item $\rho(x)$ are self-normalising for all $v\in V$, and $\rho(e)= \cap_{v\in e} \rho(v) \cap G$ for any incidence.
        \item $\rho(e)$ are self-normalising for all $e\in E$, and $\rho(v)= \cap_{e:v\in e} \rho(e) \cap G$.
        \item  $\rho(x)$ are self-normalising for all $x\in V\cup E$.
    \end{enumerate}
\end{proposition}
\begin{proof}
    We will prove the first point. The second point then follows by Lemma \ref{Dualhypergraph}, and the third point is almost immediate by the definitions.
    
    Suppose that $M= (g_{v_1}, \dots g_{i_m})$ is a motion of $\rho$, with $\rho^{M}=\rho$. Then $\rho(v)^{g_v} = \rho(v)$ for vertices $v$ implies that $g_v\in \rho_v$. Let $e \in E$, and take an arbitrary incidence $i=v*e$. If there is no such $i=v*e$, then $\rho(e) = G$, and we have $g_e\in \rho(e)$. Otherwise $\rho(i) = \rho(e)$ for any incidence, and hence we see that $g_e = g_i h_e$, where $h_e\in \rho(e)$. Thus, we see that $g_e^{-1}g_v= h_e^{-1} g_{i}^{-1} g_v\in \rho(e)\rho(v)\subseteq \rho(v)$, and hence we have $g_e\in \rho(v)$. Since $v$ was arbitrary, we see that $g_e\in \cap_{v\in e}\rho(v)$, thus $g_e\in \rho(e)$. It then follows that $g_e\in \rho(e)$. For any incidence $i=v*e$, since $g_i^{-1}g_e\in \rho(e)$, it holds that $g_{i}\in \rho(e) = \rho(i)$. 
\end{proof}

\subsection{Local and infinitesimal rigidity}\label{def Local_Inf}

We now give definitions of local and infinitesimal rigidity when the group is a Lie group. In Section \ref{sec: bar-joints}, we will show that these definition correspond to the usual definitions for bar-joint frameworks, when the group is $E(d)$. We will consider the following class of graph-of-groups realisations.
\begin{definition}
If $\rho(\Gamma)$ is a graph of groups realisation in a Lie group $G$ such that $\rho(x)$ are closed Lie subgroups for all $x\in V\cup E\cup I$, we say that $\rho(\Gamma)$ is a Lie graph-of-groups realisation in $G$. 
\end{definition}

\begin{definition}[Continuous motion]
Let $\rho(\Gamma)$ be a Lie graph-of-groups realisation in a group $G$.
A continuous motion of a graph of groups $\rho$ is a continuous curve $\gamma : [0,1] \rightarrow G^{V\cup E \cup I }$ such that $\gamma(0) = (\textup{id}, \dots,\textup{id})$ and for each $t\in [0,1]$, $\gamma(t)$ is a motion of $\rho$.
We say that two continuous motions $\gamma_1, \gamma_2$ of $\rho$ are equivalent if $\gamma_1(t)$ and $\gamma_2(t)$ are equivalent for all $t\in [0, 1]$.
\end{definition}

 Any continuous motion $\gamma$ must start in the identity, and therefore the curve $\gamma$ must lie in the identity component of the group. In this way, when working with continuous motions, one might as well restrict to the identity component of the group. 
 We now define local rigidity. The idea is that we say a graph-of-groups realisation is locally rigid if any continuous motion is equivalent to a trivial motion. 

\begin{definition}[Local rigidity]
Let $\rho(\Gamma)$ be a Lie graph-of-groups realisation in a group $G$. We say $\rho(\Gamma)$ is locally rigid if for any continuous motion $\gamma:[0, 1]\rightarrow G^{V\cup E \cup I}$, and for any $t\in [0,1]$, $\gamma(t)$ is equivalent to a congruence (i.e. a motion of the form $(g, g, \dots, g)$), for all $t\in [0,1]$.
\end{definition}

We now wish to define infinitesimal rigidity for Lie graph-of-groups realisations. Every Lie subgroup $\rho(x)$ has an associated Lie algebra $\mathfrak{h}_x\leq \mathfrak{g}$, and a Lie algebra is a vector space. To motivate the definition of an infinitesimal motion, we assume that we are given a continuous motion $\gamma$ which is also smooth. Since $\gamma(0)=(\text{id}, \dots, \text{id})$, we see that at $t=0$, the tangent vectors defined by $\gamma$ lie in $\mathfrak{g}$. Moreover, by the condition for a motion, for each vertex $x$ or edge $x$, and incidence $i$ in which $x$ participates, the tangent vector of the curve $t\mapsto \gamma_i(t)^{-1} \gamma_x(t)$ at $t=0$ satisfies
\begin{equation*}
    \frac{d(\gamma_i(t)^{-1}\gamma_x(t))}{dt}\vert_{t=0}\in \mathfrak{h}_x.
\end{equation*}
 We let $A$ be the vector space $$(v_{x_1}, \dots,v_{x_n},v_{e_1}, \dots,v_{e_m} ,v_{i_1},\dots   v_{i_k})\in \mathfrak{g}^{V\cup E \cup I},$$
such that for all $(x,i)\in (V\cup E)\times I$ with  $i=x*y$ or $i=y*x$ one has 
\begin{equation}\label{eq_infinitesimal_motion}
    d\phi_{\textup{id}}(v_i,v_x) = -v_i + v_x\in \mathfrak{h}_x,
\end{equation}
where $$\phi: G\times G\rightarrow G:(g,h)\mapsto g^{-1}h.$$ In other words, 
\begin{equation*}
\begin{array}{clll}
    A&= \underset{{i=(x,y)\in I}}{\bigcap} \left( (d(\phi\circ pr_{i,x})_{\textup{id}}^{-1}(\mathfrak{h}_x) \cap  (d(\phi\circ pr_{i,y})_{\textup{id}}^{-1}(\mathfrak{h}_y) \right),&
    \end{array}
\end{equation*}
where $pr_{i,x}:G^{V\cup E \cup I}\rightarrow G^{2}$ is the projection onto the $i$-th and $x$-th coordinate. 
If $\gamma$ and $\gamma'$ are equivalent, then for all $x\in V\cup E\cup I$ one has $\gamma_x(t)^{-1} \gamma'_x(t)\in \rho(x)$. Hence, given two $(v_{x})_{x\in V\cup E\cup I} ,(w_x)_{x\in V\cup E\cup I}\in A$, coming from two equivalent curves, one has $v_x - w_x= 0$ mod $\mathfrak{h}_x$ for all $x\in V\cup E\cup I$. For this reason, we will bound the dimension of $\pi(A)$, where $\pi$ is the map
\begin{equation*}
\begin{array}{rccl}
    \pi: &\mathfrak{g} \times \cdots \times\mathfrak{g} & \rightarrow &\mathfrak{g} /\mathfrak{h}_1 \times \cdots \times\mathfrak{g} /\mathfrak{h}_{\vert V \cup E\cup I\vert}\medskip\\
    &(v_1, \dots, v_{\vert V \cup E\cup I\vert})&\mapsto &(v_1+\mathfrak{h}_1, \dots, v_{\vert V \cup E\cup I\vert }+ \mathfrak{h}_{\vert V \cup E\cup I\vert}),
    \end{array}
\end{equation*}

Now, let
$$i: \mathfrak{g} \hookrightarrow A:v \mapsto (v, v,\dots, v).$$
This is easily seen to be a well-defined injection. Call the elements of $\pi(i(\mathfrak{g}))$  {\em trivial infinitesimal motions}. 

\begin{remark}\label{trivialmotions}
It is easy to check that
$$\ker(\pi) \cap i(\mathfrak{g})=\bigcap_{x\in V\cup E \cup I} \mathfrak{h}_x.$$ 
Thus, if $\bigcap_{x\in V\cup E\cup I} \mathfrak{h}_x=0$, one has that $\dim(\pi(A))\geq \dim(G)$. Furthermore, under the condition that $\bigcap_{x\in V\cup E\cup I} \mathfrak{h}_x=0$, one has that $\pi(A)=\pi(i(\mathfrak{g}))$ if and only if $\dim(\pi(A))=\dim(G)$.
\end{remark}

\begin{definition}[Infinitesimal motion, infinitesimal rigidity]
We call elements of $\pi(A)$ infinitesimal motions. We say $\rho$ is infinitesimally rigid if $\pi(A) = \pi(i(\mathfrak{g}))$.
\end{definition}

\begin{example}
Consider a triangle $V=\{v_1,v_2,v_3\}, E=\{e_1=v_1v_2, e_2=v_2v_3, e_{3}=v_1v_3\}$ realised in the Euclidean plane as a bar-joint framework $p$, with $p(v_1) = (0,0), p(v_2) = (1,0),$ and $p(v_3)=(0,1)$. Define a graph of groups by taking $\Stab_{E(d)}(p(v_i))$ for the vertices and $\Stab_{E(d)}(p(v_i), p(v_j))$ for the edges.

The Lie algebras of $\Stab(p(v_1)), \Stab(p(v_2))$ and $\Stab(p(v_3))$, denoted by $\mathfrak{h}_1, \mathfrak{h}_2$ and $\mathfrak{h}_3$ are given respectively by:
\begin{equation*}
    \begin{bmatrix}
        0& -t_1 &0\\
        t_1& 0 & 0\\
        0 & 0 & 0 
    \end{bmatrix},     
    \begin{bmatrix}
        0& -t_2 &0\\
        t_2& 0 & -t_2 \\
        0 & 0 & 0 
    \end{bmatrix},
     \begin{bmatrix}
        0& -t_3 &t_3\\
        t_3& 0 & 0 \\
        0 & 0 & 0 
    \end{bmatrix}
\end{equation*}
where $t_1,t_2, t_3\in \mathbb{R}$. The Lie algebras of the stabilisers of all edges are $\{0\}$. We see that for an infinitesimal motion, the algebra elements need to satisfy
\begin{equation*}
\begin{array}{ccc}
    w_{e_1*v_1}-w_{e_1} = 0, & w_{e_2*v_2}-w_{e_2} = 0, &w_{e_3*v_3}-w_{e_3} = 0,\\
    w_{e_1*v_2}-w_{e_1} = 0, &w_{e_2*v_3}-w_{e_2} = 0, &w_{e_3*v_1}-w_{e_3} =0.
\end{array}
\end{equation*}
Thus $w_{e_1} - w_{e_2} = w_{e_1 *v_2} - w_{e_2 *v_2} = w_{e_1 *v_2} -w_{v_2} + w_{v_2} - w_{e_2 *v_2},$ and hence $w_{e_1} -w_{e_2} \in \mathfrak{h}_2$. One derives in the same way that
\begin{align}\label{equations-edges1}
    w_{e_3} -w_{e_1} \in \mathfrak{h}_1,\\
    w_{e_1} -w_{e_2} \in \mathfrak{h}_2,\\
      \label{equations-edges3}  w_{e_2} -w_{e_3} \in \mathfrak{h}_3.
\end{align}
We write 
\begin{equation*}
w_{e_1}= 
 \begin{bmatrix}
        0& -a_1 &b_1\\
        a_1& 0 & c_1\\
        0 & 0 & 0 
    \end{bmatrix},    
    w_{e_2}=
    \begin{bmatrix}
        0& -a_2 &b_2\\
        a_2& 0 & c_2\\
        0 & 0 & 0 
    \end{bmatrix},    
    w_{e_3}=
    \begin{bmatrix}
        0& -a_3 &b_3\\
        a_3& 0 & c_3\\
        0 & 0 & 0 
    \end{bmatrix}.
\end{equation*}

From equations \eqref{equations-edges1} - \eqref{equations-edges3}, we derive the following equations for the coefficients:
\begin{equation*}
\begin{array}{ccc}
     -a_1 + a_2 = t_2 & b_1 - b_2 = 0 & c_1 - c_2 = -t_2\\
     -a_2 +a_3 = t_3 & b_2 - b_3 = t_3 & c_2 - c_3 = 0\\
     -a_3 +a_1 = t_1 & b_3 - b_1 = 0 & c_3 - c_1 = 0\\
\end{array}
\end{equation*}
Then $t_1 = t_2 =t_3 = 0$, and we see that we must have $w_{e_1} = w_{e_2}=w_{e_3}$. For vertices, we see that $w_{v_i} = w_{e_i} + h_{v_i}$ for some $e_i$ incident to $v$ and some $h_v \in \mathfrak{h}_i$. Hence, we see that $\pi(A)=\pi(i(\mathfrak{g})),$ and hence the graph of groups associated to the triangle is infinitesimally rigid. In Section \ref{sec:examples}, we will give a more detailed discussion for Lie graph-of-groups realisations associated to bar-joint frameworks. 
\end{example}

\subsection{A Maxwell bound for the motions of hypergraphs realised in terms of graphs of groups}\label{sec:Maxwellrk2}

As described in Section \ref{sec:bas_def_rig}, the Maxwell count for bar-joint frameworks gives a lower bound for the dimension of infinitesimal motions of a bar and joint framework. 
Inspired by this bound, in this section we give a lower bound for the dimension of the infinitesimal motions of a graph-of-groups realisation of a hypergraph. For convenience, we recall the following easy result from  linear algebra.
\begin{lemma}\label{intersection-subspaces}
If $V_1,\dots V_n$ are vector subspaces of $\mathbb{R}^{n}$, one has
\begin{enumerate}
    \item     $\dim\left(\sum_{i=1}^{k}V_i\right) \leq\sum_{i=1}^{k} \dim(V_i)$, and 
    \item     $\dim\left(\bigcap_{i=1}^{k} V_i\right) \geq n - \sum_{i=1}^{k} \left(n-\dim(V_i)\right).$\\
\end{enumerate}
\end{lemma}
\begin{proof}
We first show the first inequality. Let $\mathcal{B}_1, \dots \mathcal{B}_k$ be bases for $V_1,\dots V_k$ respectively. Then $\cup_{i=1}^{k} \mathcal{B}_i$ generates $\sum_{i=1}^{k} V_i$, and therefore we can remove some elements from $\cup_{i=1}^{k} \mathcal{B}_i$ to get a basis for $\sum_{i=1}^{k} V_i$. Thus, $\sum_{i=1}^{k} \dim(V_i)\geq \vert \cup_{i=1}^{k} \mathcal{B}_i \vert \geq \dim(\sum_{i=1}^{k}V_i)$. 

Now we will show the second inequality. Here, we will use the fact that the orthogonal complements satisfy $(\sum_{i=1}^{k} V_i)^{\perp}=\cap_{i=1}^{k} V_i^{\perp}$. Using the fact that $\dim(V_i^{\perp})= n - \dim(V_i)$, we get from the first inequality that
\begin{equation*}
     \dim(\cap_{i=1}^{k} V_i)= n - \dim(\sum_{i=1}^{k} V_i^{\perp}) \geq n - \sum_{i=1}^{k} \dim(V_i^{\perp}) = n - \sum_{i=1}^{k} (n-\dim(V_i)),
\end{equation*}
which is precisely what we wanted to show.
\end{proof}\\

We will prove a bound on the dimension of the space of infinitesimal motions. To do this, we write the space of infinitesimal motions as an intersection of vector spaces, and then use the bound on the dimension of an intersection from Lemma \ref{intersection-subspaces}. First we will prove a more general theorem that will imply our result, but which can also be used in other situations. In particular, we will use it in the proof of Theorem \ref{sparsity}. 

\begin{theorem}\label{maxwellboundLie}
Let $G$ be a Lie group. Suppose that we are given a set $\mathcal{J}$, and a closed Lie subgroup $H_x\leq G$ for each $x\in \mathcal{J}$. Denote the associated Lie algebra by $\mathfrak{h}_x$. Let $C\subset \mathcal{J}^{2}$ be a collection of ordered pairs, such that if $(x,y)\in C$, then $\mathfrak{h}_x \leq \mathfrak{h}_y$. Let 

$$\phi: G\times G\rightarrow G: (g,h)\mapsto g^{-1}h,$$
and 
\begin{equation*}
     N=\bigcap_{(x,y)\in C} d(\phi\circ pr_{x,y})_{\textup{id}}^{-1}(\mathfrak{h}_y)\leq \mathfrak{g}^{\mathcal{J}},
\end{equation*}\bigskip
where $pr_{x,y}:G^{\mathcal{J}} \rightarrow G \times G$ is the projection onto the $x$-th and $y$-th copy of G.
We have that
 \begin{equation*}
\dim(N) \geq\sum_{x\in \mathcal{J}} \dim(\mathfrak{g})- \sum_{(x,y)\in C} \dim(\mathfrak{g} /\mathfrak{h}_y),
\end{equation*}
and
\begin{equation*}
\dim(\pi(N)) \geq\sum_{x\in \mathcal{J}} \dim(\mathfrak{g} /\mathfrak{h}_x)- \sum_{(x,y)\in C} \dim(\mathfrak{g} /\mathfrak{h}_y),
\end{equation*}
where $\pi$ is the map
\begin{equation*}
\begin{array}{rccl}
    \pi: &\mathfrak{g} \times \cdots \times\mathfrak{g} & \rightarrow &\mathfrak{g} /\mathfrak{h}_1 \times \cdots \times\mathfrak{g} /\mathfrak{h}_{\vert J\vert}\medskip\\
    &(v_1, \cdots, v_{\mathcal{J}})&\mapsto &(v_1+\mathfrak{h}_1, \cdots, v_{\mathcal{J}}+ \mathfrak{h}_{\mathcal{J}}),
    \end{array}
\end{equation*}
\end{theorem}
\begin{proof}
    The map $d\phi_{\textup{id}}: \mathfrak{g} \times \mathfrak{g}\rightarrow \mathfrak{g}: (v,w) \mapsto -v +w$ is surjective, and hence by rank-nullity for any $\mathfrak{h}\subset \mathfrak{g}$, one has $\dim( (d\phi)_{\textup{id}}^{-1}(\mathfrak{h}))= \dim(\mathfrak{g}) + \dim(\mathfrak{h})$. 
Then, by rank-nullity applied to $\pi$, we furthermore have
\begin{equation*}
        \dim d(pr_{x,y})_{\textup{id}}^{-1}((d\phi)_{\textup{id}}^{-1}(\mathfrak{h})) = (n-1) \dim(\mathfrak{g}) + \dim(\mathfrak{h}).
    \end{equation*}
By Lemma \ref{intersection-subspaces}, we find
\begin{equation*}
    \dim(N)\geq\vert \mathcal{J} \vert \dim(\mathfrak{g}) + \sum_{(x,y)\in C} \left( - \dim(\mathfrak{g}) + \dim(\mathfrak{h}_y) \right).
\end{equation*}
Now, let us show $\ker(\pi) \subset N$. Let $(w_1, \cdots w_n)\in \ker(\pi)$, so $w_i \in \mathfrak{h}_i$ for every $i$. Let $(i, j)\in C$. Then we have
\begin{align*}
    d\phi_{\textup{id}} (w_i, w_j) &= -w_i +w_j \in \mathfrak{h}_j
\end{align*}
since $w_i\in \mathfrak{h}_i\subset \mathfrak{h}_j$, and $w_j\in \mathfrak{h}_j$. By definition of $N$, we conclude that $(w_1,\cdots w_n)\in N$.

Thus, since $\ker(\pi)$ is contained in $N$, we have
$$\begin{array}{rcl}
    \dim(\pi(N))&=&\dim(N) - \dim(\ker(\pi))\\\\
    &=&\dim(N) - \sum_{x\in  \mathcal{J}} \dim(\mathfrak{h}_x)\\\\
    &\geq& \sum_{x\in  \mathcal{J}}\left(\dim(\mathfrak{g})- \dim(\mathfrak{h}_x) \right)- \sum_{(x,y)\in C}  \left(\dim(\mathfrak{g}) - \dim(\mathfrak{h}_y)\right) \\\\
    &=&\sum_{x\in \mathcal{J}} \dim(\mathfrak{g} /\mathfrak{h}_x)- \sum_{(x,y)\in C} \dim(\mathfrak{g} /\mathfrak{h}_y).
\end{array}$$
\end{proof}\\

We now give a bound on the dimension of the infinitesimal motions of a Lie graphs-of-groups realisation of a hypergraph. 
\begin{theorem}\label{maxwellincLieRk2}
Let $\rho(\Gamma)$ be a Lie graph-of-groups realisation in a group $G$. Let $\pi(A)$ be the set of infinitesimal motions. Then 
\begin{align*}
\dim(\pi(A)) \geq& \sum_{x\in V\cup E} \dim(G/ \rho(x)) +\sum_{i\in I} \dim(G/ \rho(i)) \\&- \sum_{i=x*y \in I}(\dim(G/ \rho(x)) +\dim(G/ \rho(y))) 
\end{align*}

\end{theorem}
\begin{proof}
    We apply Theorem \ref{maxwellboundLie} with $\mathcal{J}= V\cup E\cup  I$ and $$C= \{(i,x)\in \mathcal{J}^{2} ~\vert~ i\text{ is an incidence with } i=x*y \text{ or } i=y*x\}.$$
    Since for any closed Lie group $H_x\leq G$, one has $\dim(G/ H_x)= \dim(\mathfrak{g}/\mathfrak{h}_x)$, this yields the theorem. 
\end{proof}\\

If the dimension of the groups $\dim\left(G/\rho(v)\right)=k_1$ is constant for $v\in V$,  $\dim\left(G/\rho(e)\right)=k_2$ is constant for $e\in E$, and $\lambda=k_1+k_2-\dim\left(G/\rho(i)\right)$ is constant for all incidences $v*e$  then Theorem \ref{maxwellincLieRk2}
can be rewritten as
\begin{equation*}
\dim(\pi(A)) \geq k_1 \vert V \vert  +k_2 \vert E \vert - \lambda \vert I \vert. 
\end{equation*}

Suppose that $\pi(i(\mathfrak{g}))= \pi(A)$.
In this case, the following bound on the incidences holds
\begin{equation*}
 \lambda \vert I \vert \geq k_1 \vert V \vert  +k_2 \vert E \vert -\dim(G).
\end{equation*}
Hence, there are at least $$\frac{ k_1 \vert V \vert  +k_2 \vert E \vert -\dim(G) }{ \lambda }$$ incidences in any hypergraph for which all motions are trivial. We will now prove that if $\rho(\Gamma)$ is infinitesimally rigid and if
$$\lambda \vert I \vert =k_1 \vert V \vert  +k_2 \vert E \vert -\dim(G),$$
then a certain sparsity condition holds on the amount of incidences. In other words, we derive a necessary condition for minimal rigidity to hold, where minimal refers to the minimal number of incidences. We will see in Section \ref{sec:examples} that from this one can retrieve the classical Maxwell bound when $G =E(d)$.
 
\begin{theorem}\label{sparsity}
    Let $\Gamma=(V,E)$ be a hypergraph, with set of incidences $I$. Let $\rho$ be a Lie graph-of-groups realisation. Assume that there are integers $k_1$, $k_2$ and $\lambda$ such that $\dim(\rho(v))=\dim(G) - k_1$ for all vertices $v\in V$, $\dim(\rho(e))=\dim(G) - k_2$ for all edges $e\in E$, and $\dim(\rho(i))=\dim(G) - k_1 -k_2 +\lambda$ for all incidences $i\in I$. 
    Suppose that $\rho(\Gamma)$ is infinitesimally rigid, and that $\Gamma$ satisfies 
    \begin{equation}
    \label{eq:assumption}
        \lambda \vert I \vert = k_1 \vert V \vert  +k_2 \vert E \vert -\dim(G).
    \end{equation} 
    Then for any subset $I'\subset I$, such that $\cap_{x\in I'} \mathfrak{h}_x = \{ 0 \}$, one has
        \begin{equation}\label{condition}
        \lambda \vert I' \vert \leq k_1 \vert V(I') \vert  +k_2 \vert E(I') \vert -\dim(G).
    \end{equation} 
    where $$V(I')=\{v\in V\:\vert \: \exists i\in I', i=v*e\}$$ and $$E(I')=\{e\in E\:\vert \: \exists i\in I', i=v*e\}.$$

\end{theorem}
\begin{proof}
We use the same notations as in the proof of Theorem \ref{maxwellboundLie}. 
Suppose for a contradiction that there exists a subset of incidences $I'$, which violates condition (\ref{condition}), so
$$\lambda \vert I'\vert > k_1 \vert V(I') \vert  +k_2 \vert E(I') \vert -\dim(G),$$
and $\cap_{x\in I'} \mathfrak{h}_x = \{ 0 \}$. 

We consider the complement $I'^c= I\setminus I'$.  Equality \eqref{eq:assumption} implies that 
\begin{align}\label{lambdabound1}
    \lambda \vert I'^c\vert &< k_1 (\vert V\vert- \vert V( I') \vert)  +k_2 (\vert E\vert - \vert E( I') \vert).
\end{align}
For an incidence, we write $x\in i$ to denote that $i=x*y$ or $i=y*x$. Now, let 
\begin{equation*}
    A_0 = \bigcap_{i=v*e\in I'} \left( d(\phi\circ pr_{v,i})_{\textup{id}}^{-1}(\mathfrak{h}_v)\cap d(\phi\circ pr_{e,i})_{\textup{id}}^{-1}(\mathfrak{h}_e)\ \right) \leq \prod_{x\in V \cup E\cup I } \mathfrak{g},
\end{equation*}
and
\begin{equation*}
    A_1 =   \bigcap_{i=v*e\in I'^c} \left( d(\phi\circ pr_{v,i})_{\textup{id}}^{-1}(\mathfrak{h}_v)\cap d(\phi\circ pr_{e,i})_{\textup{id}}^{-1}(\mathfrak{h}_e)\ \right)\leq \prod_{x\in V \cup E\cup I } \mathfrak{g}.
\end{equation*}

We first bound the dimension of $A_0$. To this end, we consider the map $i: \mathfrak{g} \rightarrow A_0 : v\mapsto (v, \dots, v)$. 
Let $$
C=\{(v_i)_{i\in V \cup E\cup I}\in \prod_{x\in V \cup E\cup I } \mathfrak{g} ~\vert ~v_i\in \mathfrak{h}_x \text{ if }x\in V(I') \cup E(I')\cup I'\}.
$$
We see by the definition of $A_0$, 
\begin{equation*}
    \textup{im}(i) + C \subset A_0.
\end{equation*}
Since \begin{equation*}
    \bigcap_{x\in I'} \mathfrak{h}_x =\{0\},
\end{equation*}
it follows that $\dim( \textup{im}(i) \cap  C )= 0$, and thus, we see that $\dim(A_0)\geq \dim(\textup{im}(i)) + \dim(C)$.
We thus have
\begin{align*} \dim(A_0) \geq  &\dim(G) + \sum_{x\in V(I') \cup E(I')\cup I' } \dim(\mathfrak{h}_x)\\ &+ \sum_{x\in (V \setminus V(I')) \cup (E \setminus E(I'))\cup I \setminus I' } \dim(G).
\end{align*}
To bound the dimension of $A_1$, we can apply Theorem \ref{maxwellboundLie} and we find
\begin{align*}
    \dim(A_1) &\geq (\vert V \vert + \vert E \vert +  \vert I \vert)\dim(G) - \sum_{i\in I'^c} \left( \dim(G) - \dim(\mathfrak{h}_v)  + \dim(G) - \dim(\mathfrak{h}_e)\right) \\
    &=  (\vert V \vert + \vert E \vert + \vert I'\vert)\dim(G) +  \dim(\mathfrak{h}_i)\vert I'^c\vert - \lambda \vert I'^c \vert 
\end{align*}
The second line follows because $$\lambda=  \dim(G) + \dim(\mathfrak{h}_i) -\dim(\mathfrak{h}_v) - \dim(\mathfrak{h}_e).$$
Then, by the bound on an intersection of two subspaces, we find
\begin{align*}
    \dim(A_0 \cap A_1) &\geq \dim(A_0) + \dim(A_1) -\left(\vert V \vert + \vert E \vert +  \vert I \vert\right)\dim(G) \\
    & \geq \dim(G) + \sum_{x\in V(I') \cup E(I') \cup I'} \dim(\mathfrak{h}_{x}) + \sum_{x\in V\setminus V(I') \cup E\setminus E(I')} \dim(G)\\
    &  - \lambda \vert I'^{c}\vert + \dim(\mathfrak{h}_i) \vert I'^{c}\vert.
\end{align*}
By inequality \eqref{lambdabound1}, and since $k_1 = \dim(G)-\dim(\mathfrak{h_v})$ and $k_2 = \dim(G)-\dim(\mathfrak{h_e})$ 
\begin{align*}
    \dim(A_0 \cap A_1) > \dim(G) + \sum_{x\in V\cup E \cup I}\dim(\mathfrak{h}_x).
\end{align*}
Thus
\begin{equation*}
    \dim(\pi(A)) = \dim(\pi(A_0\cap A_1)) > \dim(G),
\end{equation*}
which is a contradiction with the assumption that $\pi(A) = \pi(i(\mathfrak{g}))$.
\end{proof}

\subsection{Rigidity implies connectivity}

 We now establish that rigidity implies a certain degree of connectivity. Indeed, it turns out that if there exists a disconnecting subset $C$ such that $H = \bigcap_{x\in C} \rho(x) \neq \emptyset$, then one can find nontrivial motions, by picking elements from $H$ in a certain way. For graph-of-group realisations in Lie groups, it will follow that if $H$ is a Lie subgroup of dimension $\geq 1$, then there exist continuous motions of $\rho$, and $\rho$ cannot be rigid.

Let $\Gamma$ be a hypergraph. We say that a path $v\rightarrow u$ in a hypergraph is a sequence $v_0,\dots, v_n$ of vertices such that $v_i$ and $v_{i+1}$ are contained in an edge for all $i$, and $v_0=v$, $v_{n}=u$. A connected component of a hypergraph is a maximal set of vertices $V_i\subset V$ such that there exists a path between any two vertices $x,y\in V_i$.
A disconnecting set in a hypergraph is a set of vertices $V'\subseteq V$ such that there exist $v,w\in V\setminus V'$ such that any path between $v$ and $w$ contains a vertex from $V'$. If $V'\subset V$ is a disconnecting set, then the graph $\Gamma[V\setminus V']=(V\setminus V', E(V\setminus V'))$, where $E(V\setminus V') =\{e\in E ~ | ~ e\subset V\setminus V'\}$, splits into different components $C_1, \dots, C_k$, and then any edge $e\in E(V\setminus V')$ satisfies $e\subseteq E(C_j)$ for some $j$.

\begin{theorem}\label{banana}
    Suppose we are given a graph-of-groups realisation $\rho(\Gamma)$. Suppose that $V_0 \subset V$ is a disconnecting set of $\Gamma$, and that $\cap_{x\in V_0}\rho(x)\neq \{\textup{id}\}$, so there is some non-identity $\sigma \in \cap_{x\in V_0}\rho(x)$. Let $C_1, \dots C_k$ be the components of the disconnected hypergraph $\Gamma[V \setminus V_0]$. Pick some $\sigma \in \cap_{x\in V_0}\rho(x)$, and define $M=(\sigma_x)_{x\in V\cup E\cup I}$ on vertices and edges by $$\begin{cases}  
    \sigma_v=\sigma \text{ for } v \in C_1 \\
    \sigma_e=\sigma \text{ for } e \in E\text{ such that } v\in C_1 \text{ for some }v\in e\\
    \sigma_x=\textup{id}\text{ otherwise},
\end{cases}$$ 
and on incidences by $$\sigma_i = \sigma_e \text{ if } i=v*e.$$ The tuple $M$ defines a motion, and is not equivalent to a trivial motion if there is some subset of vertices $K_1 \subset C_1$ and some  $K_2 \subset C_2,\dots C_k$  such that $\bigcap_{x\in V_0 \cup K_1} \rho(x)= \bigcap_{y\in V_0 \cup K_2} \rho(y) =\{\textup{id}\}$. 
\end{theorem}
\begin{proof}
We see that $M$ is a motion since if $i=(x,e)$, then $\sigma_i^{-1}\sigma_e= \textup{id}$, and $$\sigma_i^{-1}\sigma_x=\begin{cases}
    \sigma^{-1} \text{ if } x\in V_0, i=x*e \text{ and } e\in E(V_0\cup C_1)\\
    \textup{id} \text{ otherwise}.
\end{cases}$$
Thus, in any case one has $\sigma_i^{-1}\sigma_x\in \rho(x)$.

Suppose for a contradiction that the motion was equivalent to a trivial motion given by $\sigma'\in G$, and the conditions on $K_1, K_2$ hold. Then since $\sigma'^{-1}\sigma\in \rho(v)$ for all $v\in C_1$, and since $\sigma'^{-1}\in \rho(v)$ for $v\in V_0$, we have $\sigma'^{-1}\sigma \in \bigcap_{x\in V_0 \cup K_1} \rho(x)$. Then $\sigma= \sigma'$. Since one has $\sigma'^{-1}\in \rho(v)$ for $v\in K_2 \cup V_0$, we have $\sigma' \in \bigcap_{y\in V_0 \cup K_2} \rho(y)$. Then one must have $\sigma=\textup{id}$, which is a contradiction with how the motion is defined.
\end{proof}\\

Whenever $V_0$ is as in the theorem above, and $G$ is a Lie group such that $H = \bigcap_{v\in V_0} \rho(v)$ is a Lie subgroup $\dim(H)\geq 1$, there exist nontrivial continuous motions. To see why, pick a continuous curve $\sigma: [0,1] \rightarrow H$ with $\sigma(0)=\textup{id}$, and $\sigma(t) \neq \textup{id}$ for some $t>0$. Then define a continuous by $t \mapsto (M_t)$, where $M_t$ is defined as in the theorem using $\sigma(t)$ for each $t$.

The resulting motion will not be equivalent to a congruence for some $t$. In Example \ref{Example_banana}, we apply Theorem \ref{banana} to the so-called double banana graph to explain why it is not rigid.

\subsection{Motions and sections}

We now give an interpretation of motions in terms of sections of a certain graph, equipped with a homomorphism to the incidence graph.  Although this interpretation applies also for Lie groups, the resulting graphs would have infinite degree, so the reader may find it preferable to think of finite groups.

\begin{definition}
    A graph homomorphism between two graphs $\Gamma_1=(V_1, E_1)$ and $\Gamma_2=(V_2, E_2)$ is a function $f: V_1 \rightarrow V_2$ such that if $v_1v_2$ is an edge, $f(v_1)f(v_2)$ is an edge. We denote the set of graph homomorphisms $f: \Gamma_1 \rightarrow \Gamma_2$ by $\textup{Hom}_{Grph}(\Gamma_1, \Gamma_2)$.
    
    Given a graph homomorphism $f: \Gamma_1 \rightarrow \Gamma_2$, we say that a section of $f$ is a graph homomorphism $s: \Gamma_2 \rightarrow \Gamma_1$ such that $f\circ s = \textup{id}$.
\end{definition}

Let $\rho(\Gamma)$ be a graph-of-groups realisation of a hypergraph $\Gamma = (V, E)$ in a group $G$. Consider the graph $X=(V(X),E(X))$, with vertex set $V(X)$ and edge set $E(X)$ given by:
\begin{align}
\label{vertices}V(X) &= \bigsqcup_{x\in V\cup E} G/\rho(x)\\
\label{edges}E(X) &= \bigsqcup_{i\in I} G/\rho(i)
\end{align}
where for any $i = v*e \in I$, the edge $g\rho(i)$ has endpoints $o(e)$ and $t(e)$, given by:
\begin{equation*}
\begin{array}{cc}
    o(e) = g\rho(v) & t(e) = g\rho(e).
\end{array}
\end{equation*}
This is well defined since $\rho(i) \subseteq \rho(x)$ for $x=v$ or $x=e$.

The union in (\ref{vertices}) and (\ref{edges}) is disjoint, so while one may have an equality between cosets $g\rho(x) = g' \rho(y)$, one will have two copies of a vertex appearing in the disjoint union. 
We have a map
\begin{equation*}
    q: V(X) \rightarrow V\cup E: g\rho(x) \mapsto  x.
\end{equation*}
This defines a graph morphism from $X \rightarrow I(\Gamma)$, where $I(\Gamma)$ is the incidence graph of $\Gamma$. We have an action of $G$ on $X$, given by
\begin{equation*}
    g\cdot (g'\rho(x))=gg'\rho(x).
\end{equation*}
If $\bigcap_{x\in V\cup E\cup I} \rho(x) = \{\textup{id}\}$, this action is faithful. Furthermore, the map $q:X\rightarrow I(\Gamma)$ is invariant under the action of $G$. 

\begin{theorem}\label{sections - motions}
    Let $\rho(\Gamma)$ be a graph of groups realisation, and let $X$ be constructed as above. We have a bijection of sets
    \begin{equation*}
    \{\text{Motions of } \rho(\Gamma)\}/\sim ~\longrightarrow \{s\in \textup{Hom}_{Grph}(I(\Gamma), X) ~\vert ~ q\circ s = \textup{id}\},
    \end{equation*}
    where $\sim$ denotes the equivalence of motions.
    The bijection is given by
    \begin{equation*}
    [M] \mapsto s_M
    \end{equation*}
    where $s_M$ is defined for any motion $M=(g_x)_{x\in V\cup E\cup I }$ by
    \begin{equation*}
        s_M: V\cup E \rightarrow X: x\mapsto g_x\rho(x).
    \end{equation*}
\end{theorem}
\begin{proof}
    First we show that for any motion $M$, the map $s_M:V(I(\Gamma))\rightarrow V(X)$ is a graph homomorphism.  
    To this end, let $M=(g_{x})_{x\in V\cup E\cup I}$ be a motion of $\rho(\Gamma)$. Let $i=x*y$ be an incidence. We need to show that $s_M(x)s_M(y)$ is an edge. Since the group elements $(g_x)_{x\in V\cup E\cup I}$ define a motion one has for any incidence $i= v* e$:
    \begin{equation*}
    \begin{array}{cc}
     g_i^{-1}g_v \in \rho(v), & g_i^{-1}g_e \in \rho(e).
    \end{array}
    \end{equation*}
    Thus, one has
        \begin{equation*}
    \begin{array}{cc}
        g_i\rho(v)= g_v\rho(v),&
        g_i\rho(e) = g_e\rho(e),
    \end{array}
    \end{equation*}
    and we see that $s_M(x)= g_x\rho(x),s_M(y) = g_y\rho(y)$ are the endpoints of the edge $g_i\rho(i)$. Thus, $s_M$ is a graph homomorphism. We see that $$q\circ s_M(x) = q(g_x \rho(x)) = x,$$
    and hence $s_M$ is a section.
    We now show $s_M = s_{M'}$ if and only if $M \sim M'$. Suppose that $s_M= s_{M'}$, where $M=(g_x)_{x\in P \cup L\cup I}$ and $M'=(g_x')_{x\in V \cup E\cup I}$. Then for every $x\in V\cup E$ one has
    \begin{equation*}
        g_x \rho(x) = g_{x}'\rho(x),
    \end{equation*}
    and thus $g_x^{-1}g_{x}' \in \rho(x)$ for all $x\in V\cup E$. One then also has for $i=v* e$
    \begin{align*}
        g_{i}^{-1} g_{i}' &= g_i^{-1} g_v g_v^{-1}g_v' g_v'^{-1}g_i'  \in \rho(v)\\
        g_{i}^{-1} g_{i}' &= g_i^{-1} g_{e} g_e^{-1}g_e' g_e'^{-1}g_i'  \in \rho(e),
    \end{align*}
    and thus $g_i^{-1}g_i' \in \rho(i)$. This means one has $M \sim M'$. Conversely if $M = (g_x)_{x\in V\cup E \cup I} \sim M '=(g_x')_{x\in V\cup E \cup I} $, then for all $x\in V\cup E \cup I$ one has $g_x \rho(x) = g_{x}'\rho(x)$, and hence $s_M= s_{M'}$.
    
    Finally, let us show that $\varphi$ is surjective. Given any section $s: I(\Gamma) \rightarrow X$, one gets for every $x\in V\cup E$ a coset $g_x\rho(z)$, for some $z\in  V\cup E$. One has $$x= q(s(x)) = q(g_x\rho(z))=z.$$ In this way, we get a $g_x \in G$, for every $x\in  V\cup E$. For any incidenct $v$ and $e$, $g_v \rho(v)$ and $g_e \rho(e)$ are the endpoints of an edge of $X$, thus there must exist a $g_i\rho(i)$ such that $g_i\rho(x) = g_x\rho(x)$ for both $x=v$ and $x=e$. Thus we get a group element $g_i\in G$ for each incidence, and one has $g_i^{-1}g_x\in \rho(x)$. Hence, these group elements define a motion $M:= (g_x)_{x\in V\cup E\cup I}$, and one has $s_M =s$, which shows that $\varphi$ is surjective.
\end{proof}

\begin{remark}
 The results from \cite[Chapter 1, section 5.4]{Serre1980} can be applied to relate $X$ to Bass-Serre theory, when $X$ is connected. We have seen that $G$ acts on $X$, and one has a graph isomorphism $X/G \cong I(\Gamma)$. One can check that the graph of groups constructed on $X/G$ is the original graph of groups on $I(\Gamma)$. Then, there is some normal subgroup $N$ of the fundamental group $\pi_1$ of the graph of groups such that $G \cong \pi_{1}/N$, and $X \cong T/N$ where $T$ is the universal cover of $I(\Gamma)$.
\end{remark}

\section{Examples}\label{sec:examples}

\subsection{Bar-joint frameworks}\label{sec: bar-joints}
\subsubsection{Equivalence of definitions}
Given a $d$-dimensional bar-joint framework $(\Gamma,p)$, we can associate to $(\Gamma, p)$ a graph-of-groups realisation $\rho_p(\Gamma)$ in the Euclidean group $E(d)$. We define $\rho_p$ by
\begin{equation*}
\begin{array}{clc}
    v&\mapsto \Stab_{E(d)}(p(v)) &\text{ for vertices $v\in E$}\\
    e&\mapsto  \Stab_{E(d)}(p(v)) \cap \Stab_{E(d)}(p(w)) &\text{ for edges $e=vw$}
\end{array}
\end{equation*}
For incidences, $v*e$, where $e=vw$ we get $\rho_{p}(v*e)= \rho_{p}(e)$, since $$\rho_{p}(v*e) = \rho_{p}(v)\cap \rho_{p}(e) = \rho_{p}(v)\cap \rho_{p}(v) \cap \rho_{p}(w) = \rho_{p}(v) \cap \rho_{p}(w)  =\rho_{p}(e).$$ This is a Lie graph-of-groups realisation, since all the stabilisers are closed Lie subgroups. We now make precise the relation between the standard definitions in rigidity theory, and the definitions using graphs of groups. We will need the following two facts.

\begin{lemma}\label{dist-group}
        For any $x_1,y_1, x_2, y_2\in \mathbb{R}^{d}$ one has $\Vert x_1- y_1\Vert = \Vert x_2- y_2\Vert$ if and only if there exists a $g\in E(d)$ with $x_2 =g\cdot x_1$, and $y_2 =g\cdot y_1$.
\end{lemma}
\begin{proof}
    The 'if' direction is obvious by definition, so we only need to prove 'only if'. 
    
    By applying a translation, we may assume that $x_1 =x_2 = 0$. Then it suffices to find an orthogonal matrix $T$ such that $T x_2 = y_2$. If $x_2 = 0$, then $y_2 = 0$, and there is nothing to show. Otherwise, expand $\{ \frac{x_2}{\Vert x_2 \Vert} \}$ to an (ordered) orthonormal basis $B_1$, and expand $\{ \frac{y_2}{\Vert y_2 \Vert} \}$ to an (ordered) orthonormal basis $B_2$, and let $T$ be the matrix which takes $B_1$ to $B_2$.
\end{proof}

\begin{lemma}\label{E(d)-self_normalising}
Let $d\geq 2$. Any point $x\in \mathbb{R}^{d}$ is determined by $\Stab_{E(d)}(x)$, i.e. $$\textup{Fix}(\Stab_{E(d)}(x))= \{x\}.$$
\end{lemma}
\begin{proof}
    By Lemma \ref{dist-group}, for any $x,y,z\in \mathbb{R}^{d}$, with $\Vert x - y \Vert = \Vert x - z \Vert$, there exists a $g$ such that $g \cdot x =x$, $g\cdot y=z$. The condition $g\cdot x = x$, says that $g\in \Stab_{E(d)}(x)$, and $g\cdot y=z$ then implies that $y\notin \textup{Fix}(\Stab_{E(d)}(x))$ if $y\neq z$. Since for any $y\in \mathbb{R}^{d}$ such that $y\neq x$, there exists a distinct $z\in \mathbb{R}^{d}$ with $\Vert x-y\Vert = \Vert x - z\Vert$ (for instance, $z = x - \left(y-x\right)$), it the follows that $y\notin \textup{Fix}(\Stab_{E(d)}(x))$ for all $y\in \mathbb{R}^{d}$ with $y\neq x$, and hence $\textup{Fix}(\Stab_{E(d)}(x))= \{x\}$.
\end{proof}\\

We now use these two lemmas to show that two frameworks $p_1$ and $p_2$ are equivalent if and only there exists a motion between the associated graphs of groups.

\begin{theorem}\label{equivalent_iff_motion}
Let $d\geq 2$. Suppose we are given a graph $\Gamma=(V,E)$, and two bar-joint frameworks defined by $p_1: V\rightarrow \mathbb{R}^{d}$ and $p_2: V\rightarrow \mathbb{R}^{d}$. If $(\Gamma,p_1)$ and $(\Gamma,p_2)$ are equivalent, then there exists a motion $M: \rho_{p_1}(\Gamma)\rightarrow \rho_{p_2}(\Gamma)$, where $\rho_{p_1}, \rho_{p_2}$ are defined as above in this section. Conversely, if there is a motion $M: \rho_{p_1}(\Gamma) \rightarrow \rho'(\Gamma)$, then $\rho'(\Gamma)= \rho_{p_2}(\Gamma)$ for some uniquely determined bar-joint framework $(\Gamma,p_2)$, and $(\Gamma,p_1)$ is equivalent to $(\Gamma,p_2)$.
\end{theorem}
\begin{proof}
Suppose that $(\Gamma,p_1)$ and $(\Gamma,p_2)$ are equivalent. Then for every edge $e=vw\in E$, one has $\Vert p_1(v)- p_1(w)\Vert = \Vert p_2(v)- p_2(w)\Vert$, and hence by Lemma \ref{dist-group}, for every edge $e=vw$ there exists some $g_e\in E(d)$ such that $g_e \cdot p_1(v) = p_2(v)$ and $g_e \cdot p_1(w)=p_2(w)$. We define, for any isolated vertex $v$, the translation
\begin{equation*}
    t_v: \mathbb{R}^{d} \rightarrow \mathbb{R}^{d}: x\mapsto x+(p_2(v) - p_1(v)).
\end{equation*}
For vertices which are not isolated, we choose some $e(v)$ with $v*e(v)$.
Define a motion $M= (\sigma_{x})_{x\in V\cup E \cup I}$ by 
\begin{equation*}
\begin{array}{clc}
    \sigma_v&:= g_{e(v)} \text{ for non-isolated vertices $i=v$},\\
    \sigma_v&:= t_v \text{ for isolated vertices $i=v$},\\
    \sigma_e&:= g_e \text{ for edges $i=v*e$},\\
    \sigma_i&:= g_e \text{ for incidences $i=v*e$}.
\end{array}
\end{equation*}
We verify that this is a motion of $\rho_{p_1}$ and that $\rho_{p_1}^{M} =\rho_{p_2}$. To see that $M$ is a motion, let $i=v*e$ be an incidence. Then, for the edge $e$ it holds that
\begin{equation*}
    \sigma_{i}^{-1} \sigma_e = g_e^{-1} g_e = \textup{id}.
\end{equation*}
For the vertex $v$, one has
\begin{equation*}
    \sigma_{i}^{-1} \sigma_v = g_e^{-1} g_{e(v)},
\end{equation*}
and since $g_e\cdot p_1(v) =g_{e(v)}\cdot p_1(v) = p_2(v)$, it follows that $g_e^{-1}g_{e(v)}\in \Stab_{E(d)}(p_1(v))=\rho_{p_1}(v)$. Thus $M$ is a motion. We now show $\rho_{p_1}^{M}= \rho_{p_2}$. For any $e$, with $e=vw$, one has
\begin{equation*}
     \rho_{p_1}(v)^{g_e}=\Stab(p_1(v))^{g_e}= \Stab(g_e\cdot p_1(v))=\rho_{p_2}(v),
\end{equation*}
and similarly $\rho_{p_1}(w)^{g_e}=\rho_{p_2}(w)$.
Hence, for any $v\in V$, and any $e\in E$ with
$e=vw$ it holds that
\begin{align*}
    \rho_{p_1}(v)^{\sigma_v} &= \rho_{p_2}(v),\\
    \rho_{p_1}(e)^{\sigma_e} &= \rho_{p_1}(v)^{\sigma_e} \cap \rho_{p_1}(w)^{\sigma_e}\\
    &=\rho_{p_2}(v)\cap \rho_{p_2}(w) = \rho_{p_2}(e).
\end{align*}

For the converse, we suppose we are given a motion $M: \rho_{p_1}(\Gamma) \rightarrow \rho'(\Gamma)$. We note that $\rho'(v) = \rho_{p_1}(v)^{\sigma_{v}} = \Stab_{E(d)}(\sigma_v\cdot p_1(v))$, and we thus define $p_2: V\rightarrow \mathbb{R}^{d}: v\mapsto \sigma_v\cdot p(v)$. Since $\sigma_e^{-1}\sigma_v\in \rho(v)$, one has $\rho(v)^{\sigma_e}=\rho(v)^{\sigma_v}$. Thus, for edges, $e=vw$, one has $\rho'(e)=\rho(e)^{\sigma_e} = \rho(v)^{\sigma_e} \cap \rho(w)^{\sigma_e} = \Stab(p_2(v))\cap \Stab(p_2(w))$. One thus has $\rho_{p_2} =\rho'$, and by Lemma \ref{E(d)-self_normalising} and Lemma \ref{normalizer-lemma}, $(\Gamma,p_2)$ is the only framework such that $\rho_{p_2}=\rho'$. Let us show $(\Gamma, p_2)$ and $(\Gamma, p_1)$ are equivalent. To this end, note that for any edge $e=vw$, one has 
\begin{align*}
   \rho_{p_1}(v)^{g_e}&=\Stab_{E(d)}(g_e p_1(v)) =\rho_{p_2}(v)=\Stab_{E(d)}(p_2(v))\\
    \rho_{p_1}(w)^{g_e} &= \Stab_{E(d)}(g_e p_1(w)) = \rho_{p_2}(w)=\Stab_{E(d)}(p_2(w)).
\end{align*}
Hence by Lemma \ref{E(d)-self_normalising} and Lemma \ref{normalizer-lemma}, $g_e p_1(v) = p_2(v)$ and $g_e p(w) = p_2(w)$. Thus, by Lemma \ref{dist-group}, we see that $(\Gamma,p_1)$ and $(\Gamma,p_2)$ are equivalent.
\end{proof}\\

Congruent frameworks correspond to congruent graph-of-groups realisations, giving the following result. 
\begin{corollary}\label{gl_rig_E(d)}
    A bar-joint framework $(\Gamma,p)$ is globally rigid if and only if the associated graph-of-groups realisation $\rho_{p}(\Gamma)$ in $E(d)$ is globally rigid.
\end{corollary}

 We will now show that there is a flex of $(\Gamma, p)$ if and only if there exists a continuous motion of the associated graph of groups $\rho_p(\Gamma)$. To this end we will need the following continuous variant of Lemma \ref{dist-group}.

\begin{lemma}\label{continuous-dist-group}
Let  
    \begin{align*}
        &x: [0,1] \rightarrow  \mathbb{R}^{d}
        &y: [0,1] \rightarrow  \mathbb{R}^{d}
    \end{align*}
 be continuous curves. The distance $\Vert x(t) - y(t)\Vert$ is constant if and only if there exists a continuous curve $g: [0,1] \rightarrow  E(d)$ such that $g(t)\cdot x(0) = x(t)$, and $g(t)\cdot y(0)= y(t)$, and $g(0)=\textup{id}$.
\end{lemma}
\begin{proof}
The 'if' direction is clear, so we only need to prove 'only if'.

Suppose then that we are given $x,y$ as in the statement. By applying a translation $\tau(t)$ at every $t\in [0,1]$, we may assume $x(t)=0$ for all $t\in [0,1]$, and then we replace $y$ by $y(t) - x(t)$. The map $y: [0,1] \rightarrow  \mathbb{R}^{d}$ then either has image lying in a sphere $S^{d-1}$, or $y(t)=x(t)$ for all $t$. In the latter case, we are done since then the curve is given by $t\mapsto \tau(t) - \tau(0)$.

Otherwise, we consider the map $O(d) \rightarrow S^{d-1}: g\mapsto g\cdot y(0)$. This is a fibre bundle, with fibre being the stabiliser of the point $y(0)$ by Theorem \ref{OrbitsKir} and Theorem \ref{fiber_bundle}. The proposition follows since fibre bundles have the path lifting property.
\end{proof}

\begin{theorem}\label{smooth_eq}
A continuous flex of a bar-joint framework $(\Gamma,p)$ corresponds to an equivalence class of equivalent continuous motions of $\rho_p(\Gamma)$. The bar-joint framework $(\Gamma, p)$ is locally rigid if and only if the graph of groups $\rho_p(\Gamma)$ is locally rigid.
\end{theorem}
\begin{proof}
For the statement that flexes of bar-joint frameworks correspond to continuous motions of graphs of groups, proceed as in the proof of Theorem \ref{equivalent_iff_motion}, using Lemma \ref{continuous-dist-group} instead of Lemma \ref{dist-group}. One easily checks that equivalent motions of $\rho_{p}(\Gamma)$ define the same flex of $(\Gamma, p)$.

The statement about local rigidity now follows easily. Indeed, it is clear that motions such that for all $t$, $\gamma(t)$ is equivalent to $(g, \dots, g)$ for some $g\in E(d)$, correspond to flexes that are congruences for all $t \in [0,1]$.
\end{proof}\\

We also prove that infinitesimal rigidity of $(\Gamma, p)$ is equivalent to infinitesimal rigidity of its associated bar-joint framework $\rho_p(\Gamma)$. To compute dimensions, we need a lemma.
\begin{lemma}\label{dim_stab_E}
    Let $x_1, \dots x_n\in \mathbb{R}^{d}$, which span a $k$-dimensional affine subspace $U$. Then
    \begin{equation*}
        \dim\left(\bigcap_{i=1}^{n} \Stab_{E(d)}(x_i)\right) = \binom{d-k}{2}.
    \end{equation*}
\end{lemma}
\begin{proof}
    First we show that $\bigcap_{i=1}^{n} \Stab_{E(d)}(x_i) =\bigcap_{x\in U} \Stab_{E(d)}(x_i)$. The inclusion $\supseteq$ is obvious. Suppose that $g \in \bigcap_{i=1}^{n} \Stab_{E(d)}(x_i)$, and $g \cdot x = T x +b$. Then for any affine combination $\sum_{i=1}^{n}\lambda_i x_i$ (with $\sum_{i=1}^{n} \lambda_i = 1$)
    \begin{align*}
        g \cdot (\sum_{i=1}^{n} \lambda_i x_i) &= \left( \sum_{i=1}^{n} \lambda_i T x_i\right) +b\\
        &= \sum_{i=1}^{n} \lambda_i \left(  T x_i + b\right)\\
        &= \sum_{i=1}^{n} \lambda_i x_i.
    \end{align*}
    Thus it suffices to compute the point-wise stabiliser of $U$ for some $k$-dimensional subspace $U$, and since $E(d)$ acts transitively on the $k$-subspaces, we may assume that $U=\textup{span}\{(0,\dots, 0), E_1, \dots, E_k\}$, where $E_{i} =(0, \dots,0,1,0,\dots, 0)$. Then, one can see that the Lie algebra of $\bigcap_{x\in U} \Stab_{E(d)}(x_i)$ is given by the set of all matrices of the form
    \begin{equation*}
        \begin{bmatrix}
            0_{k,k} & 0_{k, d - k} & 0_{k,1} \\
            0_{d-k,k} & S_{d-k, d-k} & 0_{d-k,1} \\ 
            0_{1,k} & 0_{1, d-k} & 0_{1,1}  
        \end{bmatrix},
    \end{equation*}
    where $0_{i,j}\in \mathbb{R}^{i\times j}$ is an all-zeros matrix, and $S_{d-k, d-k}^{t} = -S_{d-k, d-k}$. This algebra has dimension $\binom{d-k}{2}$, which completes the proof.
\end{proof}

\begin{theorem}\label{inf_eq}
Let $(\Gamma, p)$ be a bar-joint framework such that for any edge $vw$, $p(v)\neq p(w)$. Let $I=\ker(d(f_{\Gamma})_p)$ be the vector space of infinitesimal motions of $(\Gamma, p)$, and let $\pi(A)$ be the infinitesimal motions of $\rho_p$. Then,
\begin{equation*}
    \pi(A)\cong I.
\end{equation*}
Moreover, $(\Gamma, p)$ is infinitesimally rigid if and only if $\rho_p(\Gamma)$ is.
\end{theorem}
\begin{proof}
    We represent elements of $\mathfrak{e}(d)$, which is the Lie algebra of $E(d)$, by matrices
    \begin{equation*}
        \begin{bmatrix}
            S & b\\
            0 & 0
        \end{bmatrix},
    \end{equation*}
    where $S \in \mathbb{R}^{d\times d}$, and $S^{t}=-S$, and $b\in \mathbb{R}^{d}$. We consider the action of $\mathfrak{e}(d)$ on $\mathbb{R}^{d}$, which yields a vector field on $\mathbb{R}^{d}$ for all $w \in \mathfrak{e}(d)$. At a point $z\in \mathbb{R}^{d}$ the vector field is given by
    \begin{equation*}
       w\cdot z= S z + b,
    \end{equation*}
    where $S = -S^{t}$. Such vector fields define infinitesimal motions, as one can verify that for any $x,y\in \mathbb{R}^{d}$
    \begin{equation*}
        (x -y) \cdot (w\cdot x - w\cdot y) =0.
    \end{equation*}

    This implies that for any edge $e=v_i v_j$ we are given a map $\varphi_e: \mathfrak{e}(d) \rightarrow I_{e}: w \mapsto (w\cdot p(v_i) , w\cdot p(v_j))$, where $I_e$ denotes the infinitesimal motions of the single edge $e$, which has dimension $2d-1$. By Lemma \ref{dim_stab_E}, the stabiliser of an edge is an algebra of dimension $\binom{d+1}{2} - (2d - 1) = \binom{d-1}{2}$, and hence $\dim(\varphi_e(\mathfrak{g}))= 2d - 1 $, so $\varphi_e$ is surjective. Then, define a map
    \begin{equation*}
        \varphi: A \rightarrow I: (w_{v_{1}}, \dots w_{i_m}) \mapsto (w_{v_{1}} \cdot p(v))_{v\in V},
    \end{equation*}
    where $A$ is the space defined in section \ref{def Local_Inf}.
    Since for any incidence $v*e$, one has $w_{v*e} - w_{e}\in \mathfrak{h}_{e}$, and $w_{v*e} - w_{v}\in \mathfrak{h}_{v}$, and since 
    $\mathfrak{h}_e \cap \mathfrak{h}_v \subseteq \mathfrak{h}_v$ we see that
    \begin{equation}\label{eenlabel}
        w_{e} - w_{v}\in \mathfrak{h}_v,
    \end{equation}
    and hence $w_e \cdot p(v)= w_{v}\cdot p(v)$. Thus, for any edge $v_iv_j$ we see that $$\varphi(w_{v_1},\dots, w_{i_m})_{v_i}= w_{e}\cdot p(v_i) \text{ and } \varphi(w_{v_1},\dots, w_{i_m})_{v_j}= w_{e}\cdot p(v_j).$$ Hence, the morphism is well defined. We now show it is surjective. Let $(a_{v_1}, \dots, a_{v_n})\in I$ be an infinitesimal motion, with $a_{v_i}\in \mathbb{R}^{d}$ giving the components at the vertex $v_i$. For any isolated vertex $v$, pick some $w_v$ such that $w_{v}\cdot p(v) = a_{v_i}$. This exists since $\mathfrak{e}(d)/\mathfrak{h}_v \cong \mathbb{R}^{d}$. For edges $e=v_iv_j$ define $w_{e}$ to be such that $w_{e}\cdot p(v_i) = a_{v_i}$ and $w_{e}\cdot p(v_j) = a_{v_j}$, which exist since $\varphi_e$ is surjective. Then define $w_{v*e}=w_{e}$ and $w_{v}= w_{e(v)}$, where $e(v)$ is some edge having $v$ as an endpoint. It is easy to verify that this defines an infinitesimal motion $M$, and by construction $\varphi(M) = (a_{v_1}, \dots, a_{v_n})$, thus $\varphi$ is surjective.
    
    Let us compute $\ker(\varphi)$. If $\varphi(w_{v_1}, \dots, w_{i_m})=0$, then $w_v \cdot p(v) = 0$ for all $v$, and hence one has $w_v\in \mathfrak{h}_v$ for all $v$. Let $e$ be an edge with endpoints $v,u$. By equation $\eqref{eenlabel}$ for $e$ and $u,v$ we see that $w_{e}\in \mathfrak{h}_{v}\cap \mathfrak{h}_{u} = \mathfrak{h}_{e}$. Then, it similarly follows that $w_{v*e} \in \mathfrak{h}_{v*e}$, and hence we see that $\ker(\varphi) = \prod_{x\in V\cup E \cup I} \mathfrak{h}_{x}$, which shows that $\varphi$ induces an isomorphism $\pi(A)\cong I$. The claim about infinitesimal rigidity follows since $$\dim(\pi(i(\mathfrak{g}))) = 
        \binom{d+1}{2} - \binom{d - k}{2}\text{ where } k = \dim(\textup{span}(p(v))_{v\in V}),
    $$
    by Lemma \ref{dim_stab_E} and Remark \ref{trivialmotions}.
\end{proof}\\

For most of what we have done in this section, there is no problem in using the group $SE(d)$ instead of $E(d)$, since $SE(d)$ is the identity component of $E(d)$. The only exception is that Corollary \ref{gl_rig_E(d)} does not hold. One can check that Lemma \ref{continuous-dist-group} still holds when using the group $SE(d)$, and the proof of Theorem \ref{smooth_eq} and Theorem \ref{inf_eq} would be the same. When doing explicit computations for local rigidity in this model, using $SE(d)$ is simpler compared to $E(d)$.

\subsubsection{Maxwell Bound}

The bound from Theorem \ref{maxwellincLieRk2} is
\begin{equation*}
    \dim(\pi(A)) \geq d\vert V\vert + (2d-1) \vert E\vert - d\vert I \vert.
\end{equation*}
The bound from Theorem \ref{sparsity} says that for subsets $I'$ with $\cap_{i\in I'} \mathfrak{h}_i = \{0\}$
\begin{equation*}
    d \vert I'\vert \leq d\vert V(I')\vert + (2d-1) \vert E(I')\vert - \binom{d+1}{2}.
\end{equation*}
We apply this to subsets of the form $I'= V'\times E(V')$, such that $\cap_{v\in V'} \mathfrak{h}_v =\{0\}$, and $|V'|\geq d$. Then $\vert I' \vert = 2 \vert E(V')\vert$, and thus this implies 
\begin{equation*}
     \vert E(V')\vert \leq d\vert V'\vert - \binom{d+1}{2},
\end{equation*}
which is the usual Maxwell bound. The condition $\cap_{v\in V'} \mathfrak{h}_v =\{0\}$ corresponds to $p(V')$ affinely spanning $\mathbb{R}^{d}$, or a $(d-1)$-dimensional affine subspace. The following well-known example shows the bound does not give a sufficient condition for rigidity when $d\geq 3$. In the example we apply Theorem \ref{banana}.

    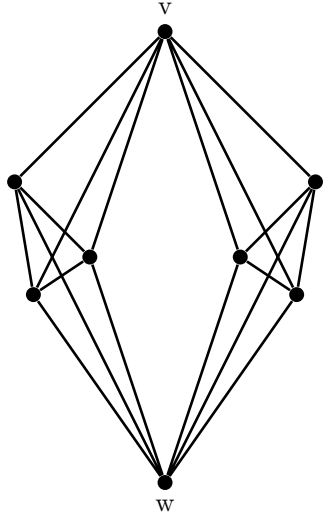
\begin{figure}[h]
    \centering
        \begin{tikzpicture}
    \node [circle,fill,inner sep=2pt][label = {v}] (V) at (0,3) {};
    \node [circle,fill,inner sep=2pt][label=below:{w}] (W) at (0,-3) {};
    \node [circle,fill,inner sep=2pt] (A) at (-1,0) {};
    \node [circle,fill,inner sep=2pt] (B) at (-2,1) {};
    \node [circle,fill,inner sep=2pt] (C) at (-1.75,-.5) {};
    \node [circle,fill,inner sep=2pt] (D) at (1,0) {};
    \node [circle,fill,inner sep=2pt] (E) at (2,1) {};
    \node [circle,fill,inner sep=2pt] (F) at (1.75,-.5) {};

    \draw [line width=1pt] (V) -- (A) -- (W) -- (B) -- (V) -- (C) -- (W);
    \draw [line width=1pt] (V) -- (D) -- (W) -- (E) -- (V) -- (F) -- (W);
    \draw [line width=1pt] (A) -- (B) -- (C) -- (A);
    \draw [line width=1pt] (D) -- (E) -- (F) -- (D);
    \end{tikzpicture}

        \caption{Double banana graph.}
        \label{fig:banana}
    \end{figure}
\begin{example}\label{Example_banana}
The graph, pictured as a framework in figure \ref{fig:banana}, called the double banana graph, provides an obstruction to generalising the Geiringer-Laman theorem to three dimensions. Namely, it satisfies the condition
\begin{align*}
   \vert E \vert &= 3\vert V\vert - 6\\
   \vert F \vert &\leq 3\vert V(F)\vert - 6 \text{ for every $F\subset E$, with $\vert F \vert \geq 3$},
\end{align*}
yet it is not generically rigid in $\mathbb{R}^{3}$.
If we consider the graph of groups $\rho_{p}(\Gamma)$, we apply Theorem \ref{banana}, with $\{v,w\}$ is a disconnecting set. The group $\rho_p(v)\cap \rho_p(w)$ contains a $1$-dimensional subgroup of rotations by Lemma \ref{dim_stab_E}. The nontrivial motions of this graph arise precisely in the way described in Theorem \ref{banana}. 
\end{example}

\subsubsection{Centres of rotation}
The condition for infinitesimal motions leads to some known facts concerning centres of motion for bar-joint frameworks in $\mathbb{R}^{2}$. See for instance \cite{PENNE2007419}, as well as \cite{RigProjlens}. In that work, it is expressed using the Grassman-Cayley algebra instead.

We note that Lie algebras of the stabiliser of a point $(x,y)\in \mathbb{R}^{2}$ is given using homogeneous matrices by
\begin{equation*}
    \begin{bmatrix}
        0 & -t & t  y \\
        t & 0 & -t x \\
        0 & 0 & 0
    \end{bmatrix}, t\in \mathbb{R}.
\end{equation*}
Given a one parameter subgroup of translations, we have the Lie algebra
\begin{equation*}
    \begin{bmatrix}
        0 & 0 & t a \\
        0 & 0 & t b \\
        0 & 0 & 0
    \end{bmatrix}, t\in \mathbb{R}
\end{equation*}
and we attach a point at infinity to this subgroup given by $[-b, a, 0].$ In this way, we have a $1-1$ correspondence between $1$-dimensional subalgebras $\mathfrak{h}_x \leq \mathfrak{e}(2)$ and points of the projective space $x\in \mathbb{R}P^2$. 
\begin{lemma}
    Let $x,y,z \in \mathbb{R}P^{2}$. Then $x,y,z$ are collinear if and only if $\mathfrak{h}_x + \mathfrak{h}_y + \mathfrak{h}_z$ has dimension $2$.
\end{lemma}
\begin{proof}
    We assume that $x=(x_1, x_2), y=(y_1, y_2), z=(z_1, z_2)$ are finite points. Other cases are dealt with similarly. The points $x,y$ and $z$ are collinear if and only if there is a nontrivial linear combination
    \begin{equation*}
        a_0 \begin{bmatrix}
            x_1 \\
            x_2 \\
            1 
        \end{bmatrix} + a_1 \begin{bmatrix}
            y_1 \\
            y_2 \\
            1 
        \end{bmatrix}  + a_2 \begin{bmatrix}
            z_1 \\
            z_2 \\
            1 
        \end{bmatrix} =0,
    \end{equation*}
     which holds if and only if
    \begin{equation*}
         \begin{bmatrix}
            0 & -a_0  & - a_0 x_2\\
            a_0 & 0  &  a_0 x_1\\
            0 & 0  &  0
         \end{bmatrix} +
         \begin{bmatrix}
            0 & -a_1 & - a_1 y_2\\
            a_1 & 0  & a_1 y_1\\
            0 & 0  &  0
         \end{bmatrix}
         + \begin{bmatrix}
            0 & -a_2 & - a_2 z_2\\
            a_2 & 0  & a_2 z_1\\
            0 & 0  &  0
         \end{bmatrix}
         =0.
    \end{equation*}
    Such a linear combination exists if and only if the algebra $\mathfrak{h}_x + \mathfrak{h}_y+ \mathfrak{h}_z$ has dimension $2$. 
\end{proof}\\

Note that for any nonzero $w_e\in \mathfrak{e}(2)$, $\langle w_e \rangle = \mathfrak{h}_x$ for some $x\in \mathbb{R}P^{2}$. Call this $x$ the centre of rotation of $w_e$. Then, a condition of the form 
\begin{equation*}
   - w_{e'} + w_{e}\in \mathfrak{h}_z \text{ if }v\in e, v\in e'
\end{equation*}
implies that $\textup{span}\{w_e, w_e'\} + \mathfrak{h}_z$ has dimension $2$. Hence, we see that the centres of rotations of $w_{e'}$ and $w_{e}$ must lie on a line with the location of $z$. This means that for any infinitesimal motion of a bar-joint framework $(\Gamma, p)$, the centres of rotations of edges $w_e$, and $w_{e'}$ must lie on a line with $p(v)$ whenever $v\in e$ and $v\in e'$. 

\subsection{Hypergraphs realised in projective space}
\label{sec:proj}

\subsubsection{Point and line configurations in the projective plane}
Let $\Gamma=(V,E)$ be a hypergraph.
Consider a realisation of $\Gamma$ as a point and line configuration in the real projective plane, so that the vertices are represented by points, and the edges are represented by lines. 

Let the Lie group of projective transformations $PGL(3,\mathbb{R})$ act upon the projective plane. The stabiliser of the line at infinity is isomorphic to the affine group $\textup{Aff}(2, \mathbb{R})$, and the stabiliser of any other line is also isomorphic to this group. By duality, the stabiliser of a point $p$ is also isomorphic to $\textup{Aff}(2, \mathbb{R})$. 
The stabiliser of the line at the infinity and a point on that line is the subgroup of $\textup{Aff}(2, \mathbb{R})$ that fixes the direction corresponding to the point, which has dimension $5$. 
The stabiliser of any pair consisting of an incident point and line pair is isomorphic to the same group, since $PGL(3,\mathbb{R})$ acts transitively on these pairs. The above can also be seen directly using coordinates.

Let $\rho$ be the graph-of-groups realisation of $\Gamma$ obtained by assigning to each $x\in X$ and to each incidence $x*y$, the stabiliser of the corresponding geometric object. 
These subgroups are all closed since they are stabilisers of a smooth group action. One can see that $\dim(H_x)=6$, $\dim(H_i)=5$, and $\dim(G)=8$, thus, applying Theorem \ref{maxwellincLieRk2}, we get the same bound as presented in \cite{proj}.

\begin{equation*}
\label{proj_bound}
\begin{array}{rcl}
    \dim(\pi(A))&\geq &\sum_{x\in V\cup E\cup I} \left(\dim(G) - \dim(\rho(x))\right) \\\\&&- \sum_{i=v*e \in I} \left( \left(\dim(G) - \dim(\rho(v))\right) + \left(\dim(G) - \dim(\rho(e))\right)\right) \\\\
    &=& \sum_{x\in V\cup E} 2 + \sum_{i\in I} 3  - \sum_{i\in I} (2+2)\\\\
    &=& 2  \vert V\vert + 2 \vert E\vert - \vert I \vert,
    \end{array}
\end{equation*}

\subsubsection{Configurations of $k$- and $\ell$-dimensional subspaces in projective $n$-space}
The above can be generalised by considering realisations inside of $PGL(n, \mathbb{R})$ with vertex groups corresponding to the stabilisers of $k-1$ and $(l-1)$-dimensional subspaces in $\mathbb{R}\mathbb{P}^{n-1}$, which we will treat as $k$ and $l$-dimensional vector subspaces of $\mathbb{R}^{n}$. 
Let $\Gamma=(V,E)$ be a hypergraph. Suppose that we are given an assignment
$$\begin{array}{cc}
\begin{array}{rccl}
    \psi_1: &V&\rightarrow &G(n,k)\\ &v&\mapsto& W_v
    \end{array}&
    \begin{array}{rccl}
    \psi_2:&E&\rightarrow& G(n,l)\\
    &e&\mapsto& U_{e},
\end{array}
\end{array}$$
where for any $m$ with $0\leq m \leq n$,  $G(n,m)$ is the set of $m$-dimensional subspaces of $\mathbb{R}^{n}$, and such that  $v\in e$ implies $\psi_1(v)\subset \psi_2(e)$.
This yields the following realisation as a graph of groups 
$$\begin{array}{cc}
\begin{array}{rccl}
   \rho: &V&\rightarrow &S(PGL(n, \mathbb{R})) \\
   &v&\mapsto &\Stab(W_v)
   \end{array}
   &
   \begin{array}{rccl}
   \rho: &E&\rightarrow & S(PGL(n, \mathbb{R})) \\
    &e&\mapsto &\Stab(U_e)
    \end{array}
\end{array}$$
The dimensions of these groups are
$$\begin{array}{lcl}
\dim(\rho(v))&=& n^{2} - k(n-k) - 1,\\
\dim(\rho(e))&=& n^{2} - l(n-l) - 1,\\
\dim(\rho(v*e))&=& n^{2} - k(n-k) - (l-k)(n-l) - 1,
\end{array}$$
 To see how to get these dimensions, we compute an example. Since $PGL(n,\mathbb{R}),$ acts transitively on $k$- dimensional linear subspaces, we may assume $W_v$ is given by $\textup{span}\{E_1,\dots E_k\}$, where $E_i$ is the $i$-th vector in a standard basis. Then the matrices which stabilise $W_v$, are given by 
\begin{equation*}
    \begin{bmatrix}
        A_{k,k} & B_{k, n-k}\\
        0_{n-k,k} & C_{n-k, n-k}
    \end{bmatrix}/\sim
\end{equation*}
Where $A_{k,k}\in Gl(k, \mathbb{R})$, $B_{k,n-k}\in\mathbb{R}^{k\times (n-k)}$ and $C\in Gl(n-k, \mathbb{R})$, and the equivalence relation is given by $M_1 \sim M_2 \iff \exists \lambda\in \mathbb{R}^{*}: \lambda M_1 = M_2$. It is clear that this manifold of matrices has dimension $n^{2} - k(n-k) - 1$. The other cases are derived similarly. 

Thus, for the bound from Theorem \ref{maxwellincLieRk2}, we find
\begin{equation*}
    \dim(\pi(A)) \geq k(n-k)\vert V\vert +  l(n-l)\vert E\vert - k(n-l)\vert I\vert.
\end{equation*}
When $n=3, k=1$ and $l=2$, we retrieve the bound for the projective plane. 

Let us make a few remarks about duality. 
 The automorphism given in projective coordinates by
$$D: PGL(n,\mathbb{R}) \rightarrow PGL(n,\mathbb{R}): A\mapsto (A^{t})^{-1},$$
is an involution. 
Let $U\leq \mathbb{R}^{n}$ be a subspace and consider the dual subspace 
$$U^{*}=\{ v \in \mathbb{R}^{n} ~ \vert ~ \langle v, x\rangle = 0 ~\forall x\in \mathbb{R}^{n}: x\in U\},$$
where $\langle \cdot, \cdot \rangle$ denotes the standard inner product. 
Then
\begin{align*}
    D(\Stab(U))&= \Stab(U^{*}),
\end{align*}
which can be seen using the identity $\langle x , A y\rangle = \langle A^{t}x , y \rangle.$ Suppose that we are given $\rho(\Gamma)$, with 
$$\begin{array}{cc}
\begin{array}{rccl}
   \rho: &V&\rightarrow &S(PGL(n, \mathbb{R})) \\
   &v&\mapsto &\Stab(W_v)
   \end{array}
   &
   \begin{array}{rccl}
   \rho: &E&\rightarrow & S(PGL(n, \mathbb{R})) \\
    &e&\mapsto &\Stab(U_e),
    \end{array}
\end{array},$$
where $W_v\in G(n,k)$ and $U_e\in G_{n,l}$. Letting $D_*(\rho)$ be as defined in Proposition \ref{GroupHom}, we see that $D_*(\rho)=\rho^{*}$, where $\rho^{*}$ is defined by
$$\begin{array}{cc}
\begin{array}{rccl}
   \rho^{*}: &V&\rightarrow &S(PGL(n, \mathbb{R})) \\
   &v&\mapsto &\Stab(W_v^{*})
   \end{array}
   &
   \begin{array}{rccl}
   \rho^*: &E&\rightarrow & S(PGL(n, \mathbb{R})) \\
    &e&\mapsto &\Stab(U_e^{*}).
    \end{array}
\end{array}$$
Applying Proposition \ref{GroupHom} and Lemma \ref{Dualhypergraph}, we see that there is a bijection between the motions $\rho(\Gamma)$ and $\rho^{*}(\Gamma^{*})$ and there is an isomorphism between the spaces of realisations of these hypergraphs. In particular, when considering configurations of points and lines in the projective plane, we see that determining the space of realisations of $\Gamma$ is the same thing as determining the space of realisations of $\Gamma^{*}$.

\subsection{Parallel redrawings and scenes}
\label{sec:parallel_redrawings_and_scenes}
As was observed by Crapo and Whiteley, the problem of determining the space of scenes over a picture is dual to the problem of determining the space of parallel redrawings with a given set of hyperplane slopes \cite{CRAPO1984278, Scenes/parallel}. We now give an explanation for this phenomenon using our approach. 
\subsubsection{Scenes}
\label{sec:scenes}
Let $\Gamma=(V,E)$, be a hypergraph. Let $Z=[0:\dots: 1]$, and define $$\pi : \mathbb{R}P^{d}\setminus \{Z\} \rightarrow \mathbb{R}P^{d-1}: [x_0:\dots: x_d]\mapsto [x_0: \dots: x_{d-1}].$$
A {\em picture} of $\Gamma$ is a function $p: V \rightarrow \mathbb{R}P^{d-1}$. A {\em scene} over the picture $p$ is a function $s: V \rightarrow \mathbb{R}P^{d-1} \setminus \{Z\}$, such that
\begin{enumerate}
    \item for every hyperedge $e$, the points $\{s(v)~\vert~ v\in e\}$ all span a hyperplane, and 
    \item for all $v$, we have $\pi(s(v)) = p(v)$. 
\end{enumerate}
In scene analysis, one is interested in determining the space of scenes over a picture. The trivial scenes, where all points lie on a single plane, always exist. We will now consider a certain group, where a subset of graph of groups realisations correspond to scenes over a fixed picture.

We consider the subgroup of the projective group $$G:=\{ A \in \textup{PGL}(d+1, \mathbb{R}) ~\vert ~\pi \circ A = \pi \}.$$ 
Using homogeneous coordinates, we can uniquely write elements of this group as
\begin{equation}
\label{eq:matrix}
    \begin{bmatrix}
    1 & 0 & \cdots & 0  \\
    0 & 1 & \cdots & 0 \\
    \vdots & \vdots &\ddots & \vdots  \\
    a_{0} & a_{1} & \cdots  &  a_{d}
   \end{bmatrix},
\end{equation}
with $a_{i}\in \mathbb{R}$ and $a_{d}\in \mathbb{R}^{*}$. 
The dimension of this group is $d+1$. The stabiliser of a point $[(x_0: x_1, \cdots: x_d)]$ consists of the subgroup of $G$ of matrices of the form (\ref{eq:matrix}) satisfying
\begin{equation}\label{stabscenes}
a_{0}x_0+a_{1}x_1+\cdots +(a_{d} - 1)x_d =0,
\end{equation}
which is of dimension $d$. The stabiliser of a hyperplane spanned by $d$ points is defined by $d$ independent equations of the form (\ref{stabscenes}), so it has dimension one.  
To any scene $s$, we associate a graph-of-groups realisation $\rho_s$ in $G$, given by
\begin{align*}
    \rho_{s}(v) &= \Stab_{G}(s(v))\\
    \rho_{s}(e) &= \Stab_{G}(\text{span}\{s(v)~\vert ~ v\in e\})= \cap_{v: v\in e} \rho(v) \\
    \rho_{s}(v*e) &= \rho(v)\cap \rho(e).
\end{align*}

One can prove using the method of proof of Theorem \ref{equivalent_iff_motion} that the resulting graphs of groups correspond to scenes over a picture, and a motion exists between two such graphs of groups if and only if the corresponding scenes are scenes over the same picture. 
 
Applying Theorem~\ref{maxwellincLieRk2}, we find the same bound as can also be found in  \cite{Scenes/parallel}
$$\begin{array}{rcl}
    \dim(A) &\geq &\sum_{v\in V} \dim(G/\rho(v)) + \sum_{e\in E} \dim(G/\rho(e))\\\\
    &&+ \sum_{i\in I} \dim(G/\rho(i)) - \sum_{i \in I} \left(\dim(G/\rho(e)) + \dim(G/\rho(v))\right)\\\\
    &=&\sum_{v\in V} 1+\sum_{e\in E} d +\sum_{i \in I} d-\sum_{i \in I} (1+ d)\\\\ 
    &=&\vert V\vert +d \vert E\vert +d\vert I \vert - d\vert I \vert -\vert I \vert=  \vert V\vert +d\vert E\vert - \vert I \vert.
\end{array}$$

\subsubsection{Parallel redrawings}\label{sec: parallel redrawings}
Given a hypergraph $\Gamma =(V,E)$, and a dimension $d$. Let $H(\mathbb{R}^{d})$ be the set of hyperplanes in $\mathbb{R}^{d}$. We define a realisation to be the triple $(\Gamma, p, h)$, where
\begin{equation*}
    p: V\rightarrow \mathbb{R}^{d}, h: E\rightarrow H(\mathbb{R}^{d}),
\end{equation*}
are such that if $v*e$, one has $p(v)\in h(e)$.
For any realisation $(\Gamma, p,h)$, the parallel redrawings are defined to be the realisations $(\Gamma, p', h')$, such that for any edge $e\in E$, $h(e)$ and $h'(e)$ have the same direction. 

To any realisation $r=(\Gamma, p, h)$, one assigns a graph of groups $\rho_r$ in the group $D$ consisting of dilations and translations
\begin{equation*}
    D = \{\phi \in \textup{Aff}(d, \mathbb{R}) ~| ~ \phi(x) = \lambda x + b, \lambda \in \mathbb{R}^{*}, b \in \mathbb{R}^{d}\}.
\end{equation*}
Note that this is also a subgroup of the projective group $PGL(d+1, \mathbb{R})$. One then sets
\begin{align*}
    \rho_{r}(e) &= \Stab_{D}(h(e))\\
    \rho_{r}(v) &= \Stab_{D}(p(v)) = \cap_{e: v\in e} \rho(e) \\
    \rho_{r}(v*e) &= \rho(v)\cap \rho(e).
\end{align*}
That is, $\rho(v)$ consists of the dilations with center $r(v)$, and $\rho_{r}(e)$ is a $d$ dimensional subgroup consisting of dilations with center at some point in $h(e)$ or translations in the direction of $h(e)$. Using matrices, we can describe the elements of $\rho_{r}(e)$ more formally, from which it directly follows that the dimension is $d$. Write
$$
P= 
\begin{bmatrix}
    \lambda & 0 & \cdots & b_1  \\
    0 & \ \lambda  & \cdots & b_2 \\
    \vdots & \vdots &\ddots & \vdots  \\
    0 & 0 & \cdots  &  1,
\end{bmatrix},
$$
with $\lambda \in \mathbb{R}^{*}, b_1, \cdots b_d\in \mathbb{R}$.
Suppose that $h(e)$ is given by the equation 
$$\alpha_1 x_1 +\cdots \alpha_d x_d = c,$$
with $[(\alpha_{1}: \dots: \alpha_d: -c)] \neq [(0:\dots:0:1)]$.

Then, the elements of $\Stab(h(e))$ are precisely those matrices $P$ such that
\begin{equation*}
    P^{t} \begin{bmatrix}
        \alpha_1\\
        \alpha_2\\
         \vdots \\
         \alpha_{d} \\
         -c
    \end{bmatrix}=\mu \begin{bmatrix}
        \alpha_1\\
        \alpha_2\\
         \vdots \\
         \alpha_{d} \\
         -c
    \end{bmatrix},
\end{equation*}
for some $\mu \in \mathbb{R}^{*}$. One verifies that these are the matrices $P$ which satisfy
\begin{equation*}
    b_1 \alpha_1 + b_2 \alpha_2 + \cdots + b_d\alpha_d + (\lambda - 1) c =0.
\end{equation*}

In the same way as in Theorem \ref{equivalent_iff_motion}, one can prove that the parallel redrawings are in 1-1 correspondence with their associated graphs of groups.

The Maxwell bound from Theorem \ref{maxwellincLieRk2} becomes
\begin{equation*}
   \pi(A) \geq d \vert V\vert +\vert E\vert - \vert I \vert.
\end{equation*}

\subsubsection{Duality between scenes and parallel redrawings}

We now show using our model that these two problems are dual to each other. We consider the group $G$ defined in section \ref{sec:scenes} and the group $D$ defined in \ref{sec: parallel redrawings} as subgroups of the projective group. We see that 
\begin{equation*}
\begin{bmatrix}
    1 & 0 & \cdots & 0  \\
    0 & 1 & \cdots & 0 \\
    \vdots & \vdots &\ddots & \vdots  \\
    a_{0} & a_{1} & \cdots  &  a_{d}
\end{bmatrix}^{T}
= \begin{bmatrix}
    1 & 0 & \cdots & a_0  \\
    0 & 1 & \cdots & a_1 \\
    \vdots & \vdots &\ddots & \vdots  \\
    0 & 0 & \cdots  &  a_{d}
\end{bmatrix}
\sim \begin{bmatrix}
    \frac{1}{a_{d}} & 0 & \cdots & a_0  \\
    0 & \frac{1}{a_{d}}  & \cdots & a_1 \\
    \vdots & \vdots &\ddots & \vdots  \\
    0 & 0 & \cdots  &  1
\end{bmatrix},
\end{equation*}
and hence transposing elements of $G$ yields of $D$
since the last matrix is clearly an element of $D$, and the first is a transpose of an element of $G$. In this way, we see that
\begin{equation*}
    \psi: G\rightarrow D: A\mapsto (A^{t})^{-1}
\end{equation*}
is a group isomorphism. Moreover, if $H$ is the hyperplane defined by
\begin{equation}\label{eq_Hyperplane}
\alpha_1x_1 + \cdots  + \alpha_d x_d = c, \text{ with } \alpha_i\neq 0, \text{ for some } i \in \{1,\cdots d\}
\end{equation}

and if $p = [\alpha_1: \cdots :\alpha_d : -c] \neq Z$, is the polar of $H,$ then we see that:
$$\psi(\Stab_{G}(p)) = \Stab_{D}(H).$$
We had defined $\rho_s(v) = \Stab_{G}(s(v))$ for $s(v)\in \mathbb{R}P^{d} \setminus \{Z\}$ and $\rho_r(e) = \Stab_{D}(h(e))$ for a some hyperplane $h(e)$ defined by an equation \eqref{eq_Hyperplane}. Furthermore, we had $\rho_{s}(e)= \cap_{v\in e}\rho_s(v)$ and $\rho_{r}(v)= \cap_{e: v\in e}\rho_r(e)$. Thus, by Lemma \ref{Dualhypergraph} and \ref{GroupHom}, it follows that one has a isomorphism of groupoids, and hence one has a one-to-one correspondence between graph-of-groups realisations corresponding to scenes over a given picture and graph-of-groups realisations corresponding to parallel redrawings to a given realisation. It is in this sense that scenes and parallel redrawings are dual problems. 

\subsection{Linearly constrained frameworks (and rigidity of frameworks drawn on surfaces}\label{sec:linearcons}

A linearly constrained bar-joint framework is a framework $p:V\rightarrow \mathbb{R}^{d}$ of a graph $\Gamma=(V,E)$, such that each point realising a vertex $v\in V$ is restricted to move in an affine subspace $L_v$ \cite{Linearconstraints}. 
A motivation for studying linearly constrained bar-joint frameworks comes from the theory of rigidity of frameworks on surfaces, or more generally, frameworks on embedded submanifolds in $\mathbb{R}^{d}$. By constraining the motions to lie in a certain subgroups, we can describe motions of linearly constrained frameworks using graphs of groups as well.

Given a bar-joint framework $(\Gamma,p)$, we define a graph-of-groups realisation $\rho_p(\Gamma)$, as in section \ref{sec: bar-joints}, by 
 $$\begin{array}{rccll}
 \rho_p: &\Gamma&\rightarrow&S(E(d))\\
 &v& \mapsto & \Stab(p(v)) &\mbox{ for }v\in V  \\
 &e & \mapsto &\Stab(p(e)) &\mbox{ for }e\in E\\
 &v*e & \mapsto &\Stab(p(e)) \cap \Stab(p(v)) &\mbox{ for }v*e\in I 
 \end{array}$$
Now, let $$H_v = \{g\in \text{SE}(d) ~\vert~ g\cdot L_v = L_v \},$$ which is a closed Lie group with Lie algebra $\mathfrak{h}_v$. If we consider a motion $\{\sigma_{x}\}_{x\in V\cup E\cup I}$ of the realisation $\rho$, such that $\sigma_v\in H_v$ for every vertex $v$, this precisely corresponds to a motion which respects the line constraints. As in Theorem \ref{maxwellincLieRk2}, we can find a lower bound on the dimension of infinitesimal motions. Let 
\begin{equation}
\begin{array}{cl}
 A=&\{(v_{x_1}, \dots v_{x_n}, v_{i_1} \dots v_{i_m}) \in \mathfrak{h}_{v_1} \times \cdots \times \mathfrak{h}_{v_{\vert V\vert}} \times  \mathfrak{e}(d)^{E +I} \\
 &\text{ for every element-incidence pair} (x,i) \text{ one has } -v_i + v_x\in \mathfrak{g}_x \},
\end{array}
\end{equation}
where $\mathfrak{g}_x$ is the Lie algebra of $\rho(x)$, and $\mathfrak{e}(d)$ is the Lie algebra of $E(d)$. Using the same method of proof as in Theorem \ref{maxwellboundLie} it follows that 
\begin{equation}\label{constrainedbound}
\begin{array}{cl}
    \dim(\pi(A)) \geq &\sum_{v\in V} \dim\left(H_v/(\rho_p(v)\cap H_v)\right) + \sum_{e\in E} \dim\left(G/(\rho_p(e)\right)\\ 
    &\\
    &+ \sum_{i\in I} \dim\left(G/(\rho_p(i)\right)
 -  \sum_{i\in I} \left( \dim\left(G/(\rho_p(v)\right)  + \dim\left( G/(\rho_p(e)\right)\right)
\end{array}
\end{equation}

where $\pi$ is defined as
$$\pi: \mathfrak{g} \times \mathfrak{g} \times \cdots \times \mathfrak{g}  \rightarrow \mathfrak{g} /\mathfrak{g}_{x_1} \times \mathfrak{g}/\mathfrak{g}_{x_2} \times \cdots \times  \mathfrak{g} /\mathfrak{g}_{i_{m}}.$$

Equation \ref{constrainedbound} can be written as 
\begin{align}
    \dim(\pi(A)) \geq &\sum_{v\in V} \dim(L_v) + (2d-1) \vert E \vert + ( 2d -1 )\vert I\vert -  (3d - 1) \vert I\vert.
\end{align}
Using $I=2\vert E\vert $, one has
\begin{align}
    \dim(\pi(A)) \geq &\sum_{v\in V} \dim(L_v) - \vert E \vert.
\end{align}
The fact that $\dim(L_v) = \dim(H_v/(\rho_p(v)\cap H_v))$ follows since $H_v/(\rho_p(v)\cap H_v)$ is diffeomorphic to the orbit of the point $v$ under the group action of $H_v$, which is the linear space to which $v$ is constrained.

Although stated differently, in \cite{Linearconstraints}, the authors give essentially the same bound, and furthermore a Geiringer-Laman type result is shown for graphs in which there are at least $m(d)$ linear constraints at each vertex, where $m(d)$ depends on $d$.

\subsection{Graphs realised in the complete graph on $n$ vertices and the unique colourability of graphs as a (global) rigidity problem}\label{sec: graph_colourings}
Let $G= S_n$ be the symmetric group acting on the complete graph on $n$ vertices, and let $\Gamma= (V,E)$ be a graph. We will show that certain realisations of $\Gamma$ as a graph of group in $S_n$ correspond to proper colourings of the graph $\Gamma$, and that the realisation is globally rigid if and only if the underlying graph is uniquely $n$-colourable.

A proper $n$-colouring of $\Gamma$ is an assignment of a colour to each vertex $v\in V$ such that no two adjacent vertices get the same colour. This corresponds to giving a graph homomorphism from $\Gamma$ to the complete graph $c: \Gamma \rightarrow K_n$. A graph is uniquely $n$-colourable, if for any two homomorphisms $c,c': \Gamma \rightarrow K_n$, there exists a $\phi\in \text{Aut}(K_n)\cong S_n$ such that $\phi \circ c = c'$. 

Let $c$ be a colouring and consider the associated realisation of $\Gamma$ as a graph of groups
\begin{align*}
    &\rho(v)= \Stab(c(v)).\\
    &\rho(e)=\rho(v)\cap \rho(w). 
\end{align*} 
Note that this implies that inverting the colours in an edge does not fix the edge. The following result can be proved by using the same proof technique as in Theorem \ref{equivalent_iff_motion}. 
\begin{proposition}
Let $\Gamma$ be a graph. For any proper $n$-colouring $c$, with $n\geq 3$, let $\rho_c:\Gamma\rightarrow S(S_n)$ be the corresponding graph-of-groups realisation of $\Gamma$ as described above. 
Let $M=\{\sigma_x\}_{x\in V\cup E \cup I}$ be a motion of $\rho_c$. Then $\rho(\Gamma)^M=\rho_{c'}$ for some uniquely determined proper colouring $c'$. For any pair of proper colourings $c$, $c'$ with respective graph of group realisations $\rho_c$ and $\rho_{c'}$, there exists a motion $M$ such that $\rho_c^{M}=\rho_{c'}$
\end{proposition} 

Furthermore, it is clear that two such realisations are congruent if and only if there exists some $\sigma \in S_n$ such that the two colourings are the same up to $\sigma$. From this we conclude that the globally rigid graphs are precisely the uniquely colourable graphs. 

Some of the basic properties of uniquely colourable graphs are very reminiscent of those of rigid graphs for the Euclidean group, as we now briefly summarise.

Except when stated otherwise, the following properties were pointed out by Harary, Hedetniemi, and Robinson \cite{HARARY1969264}.  If one takes the vertices of two different colours $$V_{ij}=\{v\in V ~\vert ~p(v)=i\text{ or } p(v)=j \},$$ then the induced subgraph $\Gamma[V_{ij}]$ is connected. To see why, notice that if this were not the case, one could swap out the colours $i$ and $j$ within one component, which would yield a nonequivalent colouring. Thus, for any uniquely colourable graph, there exists a subgraph which has an edge decomposition into $\binom{k}{2}$ trees such that any vertex $v$ is included into exactly $k-1$ trees. For $k=3$, the existence of such a tree decomposition implies  generic rigidity in the plane \cite{crapo1990generic}. From this decomposition, it follows that one has the following bound on the edges \cite{XUSHAOJI1990319}.
\begin{equation*}
    \vert E\vert \geq (k-1)\vert V \vert - \binom{k}{2}.
\end{equation*}
Moreover, if one cones a graph which is uniquely $k$-colourable, one gets a graph which is uniquely $k+1$ colourable. Coning is a procedure where one adds one vertex to  a graph $\Gamma$, adding edges to all other vertices, resulting in a new graph $\hat{\Gamma}$. This operation is known to have the property that $\Gamma$ is generically rigid in $d$-dimension if and only if $\hat{\Gamma}$ is generically rigid in $d+1$ dimensions, and similar results hold for global rigidity \cite{Connelly2009}. Also, if we take add a vertex to a uniquely $k$ colourable graph, connected to $k-1$ colour classes, the result is again uniquely colourable, which is analogous to the well known $0$-extension from rigidity theory.

This example can be generalised as follows. Let $\Lambda$ be a graph. A graph is uniquely $\Lambda$-colourable if there exists a unique graph homomorphism up to an automorphism of $\Lambda$. So, $\Gamma$ is uniquely $\Lambda$-colourable if for any two graph homomorphisms $f,g: \Gamma\rightarrow \Lambda$  there exists an automorphism of $\alpha: \Lambda\rightarrow \Lambda$ such that $f= \alpha \circ g$.  
When $\Lambda$ is an arc-transitive graph, this corresponds to rigidity of $\Gamma$ in $\text{Aut}(\Lambda)$. This problem was studied for odd cycles in the 80's \cite{lai1985unique}, \cite{LAI1989363}, and also more recently \cite{noteUniquecolorability}. 
\begin{example}
Theorem \ref{sections - motions} has a straightforward interepretation in this context. Suppose that we want to study $\{f: \Gamma \rightarrow \Lambda\}$, where $\Lambda$ is an arc-transitive graph. Let $X=(V,E)$ be the tensor product of $\Lambda$ and $\Gamma$ (also called categorical product or Kronecker product). It is defined by:
    \begin{align*} 
        V &= V(\Gamma) \times V(\Lambda)\\
        E &= \{(x_1, y_1), (x_2, y_2) ~\vert~ x_1x_2 \in E(\Gamma), y_1y_2 \in E(\Lambda)\}.
    \end{align*}
     Let
    $$q: V(X) \rightarrow V(\Gamma): (x, y)\mapsto x.$$
    The reader can verify that sections $s: \Gamma \rightarrow X$ of $q$ correspond to graph homomorphisms into $\Lambda$.
    Technically, the graph $X$ as constructed in Theorem $\ref{sections - motions}$ is isomorphic to the incidence graph of $\Lambda \times \Gamma$, and one would take the incidence graph $I(\Gamma)$ as well.
\end{example}
\section{Conclusions and open problems}
We have considered structural rigidity and flexibility in a purely group-theoretic setting, and established the basic terminology in terms of graphs of groups. We have showed that the model that we have presented applies to various problems in constraint geometry, and we now state some open problems.

\begin{enumerate}
    \item (Genericity) It would be interesting to describe precisely what the connection is between local and infinitesimal rigidity, as well as to explore when rigidity is a generic property. With respect to genericity, one may desire that 'almost all' graph-of-groups realisations of a hypergraph $\Gamma$ in a group $G$ have the same rigidity behaviour, so that either almost all realisations are rigid, or almost all realisations are flexible. The main difficulty then, is to suitable interpretation of the term 'almost all' in this context. 
    \item (Characterising rigidity) Whenever rigidity is a generic property of the graph, one can ask when the necessary condition in Theorem \ref{sparsity} is a sufficient condition for minimal rigidity, in terms of the subgroups of the graph. That is, under which conditions does the analogue of the Geiringer-Laman theorem hold? To answer such questions, it might be useful to look at inductive constructions of Henneberg type.
    \item (Connections to geometric group theory) It would also be interesting to explore further the connection with Bass-Serre theory. One question is: What information can be found in the subgroup $N$ as defined at the end of Section \ref{sec:Defs_hypergraphs}? 
\end{enumerate}

\bibliography{Bibliography}

\end{document}